\documentclass[pdftex,a4paper,12pt]{article}

\usepackage[english]{babel}
\usepackage[latin1]{inputenc}
\usepackage{amsmath}
\usepackage{amsthm}
\usepackage{amssymb}
\usepackage{mathtools}
\usepackage{enumerate}
\usepackage{hyperref}
\usepackage{graphicx}
\usepackage[table]{xcolor}

\newcommand{\NN}{\mathbb{N}} 
\newcommand{\ZZ}{\mathbb{Z}} 
\newcommand{\RR}{\mathbb{R}} 

\newcommand{\mv}{{\mathsf{m}}}
\newcommand{\nv}{{\mathsf{n}}}

\newcommand{\myC}{\mathcal{CC}}
\newcommand{\perco}[7]{\text{\normalshape X}_{#1}([#2,#3,#4\!\mid\! #5,#6,#7])}
\definecolor{lgray}{rgb}{0.9,0.9,0.9}
\newcommand{\nsp}[1]{\cellcolor{lgray}{#1}}
\definecolor{dgray}{rgb}{0,0.9,0.9}
\newcommand{\nspi}[1]{\cellcolor{dgray}{#1}}
\definecolor{llgray}{rgb}{0.9,0,0.9}
\newcommand{\nspii}[1]{\cellcolor{llgray}{#1}}
\newcommand{\id}{\operatorname{id}}
\newcommand{\CG}{\mathcal{CG}}

\newtheorem{lemma}{Lemma}[section]
\newtheorem{propo}[lemma]{Proposition}
\newtheorem{theor}[lemma]{Theorem}
\newtheorem{corol}[lemma]{Corollary}

\theoremstyle{definition}
\newtheorem{defin}[lemma]{Definition}
\newtheorem{remar}[lemma]{Remark}
\newtheorem{quest}[lemma]{Question}

\usepackage{algorithm}
\usepackage{algorithmicx}
\usepackage{algpseudocode}
\newenvironment{proofof}[1]{\begin{proof}[Proof of #1]}{\end{proof}}

\author{S\o{}ren Eilers, Rune Johansen, Rasmus Veber Rasmussen\\and Carsten Thomassen}
\title{Chromatic numbers for contact graphs of congruent cuboids}
\date{\today}
\begin{document}

\maketitle
\begin{abstract} We initiate the study of chromatic numbers for contact graphs of configurations of integer-sized cuboids in three dimensions, all of which are mutually congruent. Disallowing rotations, we show a global upper bound of 8 for the chromatic numbers, which implies that there is a global upper bound of 48 when the cuboids may be rotated freely. Specializing further to cuboids that are required to have a side length of one we obtain more precise upper bounds.

Such upper bounds are compared to examples of configurations having relatively large chromatic numbers, leading to a complete determination of some of these chromatic numbers, but in general, the gaps between our upper and lower bounds are rather wide. In particular, we know of no such configuration of any size leading to a chromatic number above 6.
\end{abstract}
\section{Fundamentals}

\subsection{Introduction}

The study of chromatic numbers of contact graphs for axis-parallel cuboids goes back at least to the Study Group with Industry in 2005 (retroactively published in \cite{Allwright_2021}), and it has been known since \cite{brda:po} (see also \cite{MR2809403}) that there is no bound for the chromatic number of such contact graphs. In a related endeavor, it was shown in \cite{bessygoncalvessereni} that there are layered configurations needing eight colors, and since one may appeal to the four color theorem in alternating layers to show that eight colors suffice, this determines the optimal chromatic number in that case as well. 

We propose here the study of chromatic numbers for contact graphs of cuboid configurations where all constituent cuboids are mutually congruent. As a first example, consider 
12 cuboids of size $8\times 2\times 1$ placed like 
\begin{center}
\includegraphics[width=6cm]{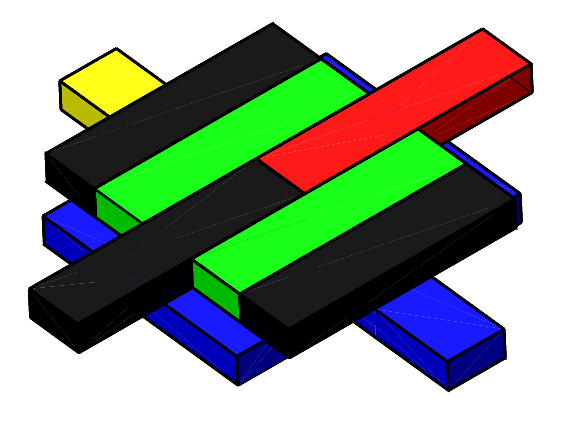}
\end{center}
or, in exploded view,
\begin{center}
\includegraphics[width=6cm]{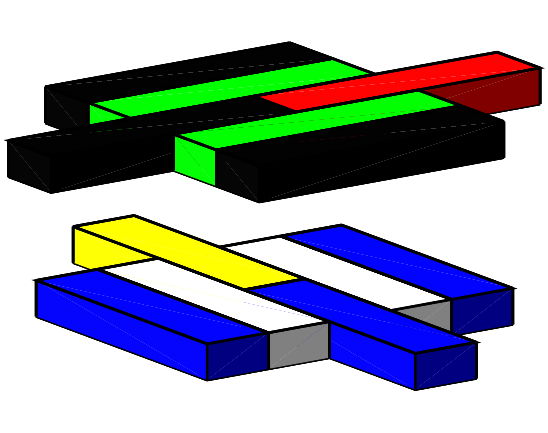}
\end{center}
It is easy to check that the given 6-coloring is proper, and if one starts by noting that the four cuboids in the center of the configuration need to be colored with four colors, it is straightforward to go over the remaining number of choices to see that the configuration cannot be colored with just 5 colors. Note that the configuration is layered but that it is constructed from translates of two differently oriented cuboids. 

We will present examples with the same chromatic number of 6 where the cuboids are all translates of each other, and/or where the cuboids are smaller, but it must be pointed out  that we know of no example with congruent cuboids of any size provably requiring 7 colors. 
The examples shown in \cite{brda:po} or \cite{MR2809403} to render the chromatic number unbounded in the general case, and in \cite{bessygoncalvessereni} to establish tightness of the upper bound in the layered case, requires repeated rescaling of the cuboids and thus we see no way to employ such methods to provide lower bounds. Indeed, except for the example already given, and the one presented in Remark \ref{byhand}, all of the lower bounds presented here were obtained by locating examples in pseudorandom computer searches. Some of these examples are so large that it is far outside the scope of human computation to check that the stipulated chromatic numbers are correct. As further explained in Section \ref{pseudo} we are using standard and well-tested  packages, and in some cases we have double checked our findings using alternative (and slower) software, always obtaining coherent answers. For the benefit of the reader who may wish to independently confirm our results, and/or visualize the found configurations in a way allowing rotation, we list coordinates with a coloring requiring the least possible number of colors in Appendix A. 

We obtain upper bounds by a variety of methods, the most effective of which is to locate periodic colorings of $\RR^3$ that may then be inherited to colorings on cuboid configurations. This does not lead to the exact computation of the chromatic numbers of cuboids beside $1\times 1\times 1$ in general position, but when we consider cuboid configurations that are layered and/or where the cuboids are never rotated, we do know some chromatic numbers exactly.

The key case where the cuboids are of the form $a\times b\times 1$ and may only be rotated in the $XY$ plane may be interpreted as a question on LEGO bricks, which partially motivated this work as an outgrowth of the first listed author's endeavors in this direction (\cite{legocprob,durhuuseilers}).
The study of these types of problems was featured in the textbook \cite{xmbook} by the two authors first listed, and for several years students in the ``Experimental Mathematics'' course at the University of Copenhagen would attack a selection of such  problems with computer-based experimentation. Several  ideas for the results presented in this paper appeared in embryonic form in the students' course work. It is difficult to track the development of these ideas and attribute them to individuals, but we can very precisely do this for milestones of upper and lower bounds, and will do so in remarks below. It is  also a pleasure to record our gratitude to Micha\l\ Adamaszek and Eigil Rischel for inspiring conversations in the earlier phases of this work.

\subsection{Definitions}

We are interested in contact graphs for cuboids, and   study the case when all cuboids in a configuration are mutually {congruent}, of integer dimensions, and with corners in $\ZZ^3$. More precisely, we consider the sets 
\begin{eqnarray*}
\myC_1([a,b,c])&=&\{[x,x+a]\times [y,y+b]\times [z,z+c]\mid x,y,z\in \ZZ\}\\
\myC_2([a,b,c])&=&\myC_1([a,b,c])\cup \myC_1([b,a,c])\\
\myC_3([a,b,c])&=&\myC_2([a,b,c])\cup \myC_2([a,c,b])\cup \myC_2([b,c,a])\\
\end{eqnarray*}
where the subscript indicates the freedom we allow ourselves in rotating the cuboids. The tuple $(x,y,z)$ is called the \emph{root} of the cuboid in all cases.

With $\bullet\in\{1,2,3\}$ and for any set  $\{C_1,\cdots,C_m\}\subseteq \myC_\bullet([a,b,c])$ with $C_i^\circ\cap C_j^\circ=\emptyset$ for $i\not=j$, we define the \emph{contact graph} with vertices $C_i$ and edges connecting $C_i$ and $C_j$, $i\not=j$, when they \emph{touch}, i.e., when  their intersection is a non-degenerate rectangle, or, which is the same, when
\[
C_i\cap C_j\not\subseteq (\ZZ\times \ZZ\times \RR) \cup  (\ZZ\times \RR\times \ZZ) \cup  (\RR\times \ZZ\times \ZZ).
\]

 The set $\CG_\bullet([a,b,c])$ is the set of graphs thus obtained. Our main object of study is then

\begin{defin} For $\bullet\in\{1,2,3\}$, 
\[
\chi_\bullet([a,b,c])=
\max\{\chi(G)\mid G \in \CG_\bullet([a,b,c])\}
\]
\end{defin}

In other words, $\chi_\bullet([a,b,c])$ is the smallest number of colors needed to ensure that any finite configuration of specified cuboids (no two of which have an interior point in common) can be colored with no touching cuboids having the same color.

\section{Methodology}
In this section, we summarize various methods to bound the numbers $\chi_\bullet([a,b,c])$, and  introduce further notation.

\subsection{Rescaling}\label{rescaling}

All the previously mentioned results  \cite{brda:po},  \cite{MR2809403}, \cite{bessygoncalvessereni}, proving existence of cuboid configurations of varying dimensions with high chromatic numbers are based on the careful rescaling of various ``gadgets'' to arrange systematic contacts that can then be used to show that many colors are needed. This method is not available in our setting, since the cuboids' dimensions are given. But we can show a fundamental result describing the action of rescaling all  cuboids in a given configuration.

To do so, and for later use, we write out what it means that two cuboids $[x,x+a]\times [y,y+b]\times [z,z+c]$ and $[x',x'+a']\times [y',y'+b']\times [z',z'+c']$ touch without colliding, and consequently induce an edge in a graph under study. One of the following statements must hold for this; either 
\begin{gather}
(x+a=x'\vee x'+a'=x)\wedge\notag\\
  (y+b\geq y'\wedge y'+b'\geq y) \wedge\label{touchx}\\
(z+c\geq z'\wedge z'+c'\geq z)\notag
\end{gather}
or
\begin{gather}
  (x+a\geq x'\wedge x'+a'\geq x) \wedge\notag\\
(y+b=y'\vee y'+b'=y)\wedge\label{touchy}\\
(z+c\geq z'\wedge z'+c'\geq z)\notag
\end{gather}
or
\begin{gather}
  (x+a\geq x'\wedge x'+a'\geq x) \wedge\notag\\
  (y+b\geq y'\wedge y'+b'\geq y) \wedge\label{touchz}\\
(z+c=z'\vee z'+c'=z)\notag
\end{gather}
depending upon in which direction the blocks touch.

Then we prove

\begin{lemma}\label{increases}
When $[a,b,c]\leq [a',b',c']$ then $\CG_1([a,b,c])\subseteq \CG_1([a',b',c'])$.
\end{lemma}
\begin{proof}
We rescale by
\[
(xa+r,yb+s,zc+t)\mapsto (xa'+r,yb'+s,zc'+t)
\]
where $0\leq r<a,0\leq s<b, 0\leq t<c$. This will preserve the contact graphs by reference to one of the statements \eqref{touchx}--\eqref{touchz}. For instance, if two blocks rooted at $(xa+r,yb+s,zc+t)$ and $(x'a+r',y'b+s',z'c+t')$, respectively, touch as in \eqref{touchx} relative to the $a\times b\times c$ dimensions, we may assume by symmetry that
\[
xa+r+a=x'a+r'
\]
which implies $x'=x+1$ and $r'=r$, and consequently
\[
xa'+r+a'=x'a'+r'.
\]
We also have
  \[
  yb+s+b\geq y'b+s'\wedge y'b+s'+b\geq yb+s 
  \] 
which implies $-1\leq y-y'\leq 1$ and is equivalent to
  \[
  (y-y'+1)b\geq s'-s\wedge (y'-y+1)b\geq s-s'. 
  \] 
  Since the coefficients on the $b$ are non-negative, we also get
  \[
  (y-y'+1)b'\geq s'-s\wedge (y'-y+1)b'\geq s-s' 
  \] 
  because $b'\geq b$. The same argument works for $z$ and $t$, and as above in the appropriate order when \eqref{touchy} or \eqref{touchz} are in effect.
  \end{proof}

\subsection{Partitions}\label{partitions}

As already noted, we have $\chi_2([a,b,1])\leq 8$ for all $a,b$ by partitioning
\[
\CG_2([a,b,1])=\CG_2^{(0)}([a,b,1])\sqcup \CG_2^{(1)}([a,b,1])\
\]
with $\CG_2^{(i)}([a,b,1])$ the elements rooted in $(x,y,z)$ with $z$ even or odd, respectively. The four color theorem shows that every configuration rooted at a fixed $z$ can be four-colored, and since the layers in $\CG_2^{(i)}([a,b,1])$ do not touch, the same is true throughout. 

%
Similar reasoning shows:

\begin{lemma}\label{partitioning}
For all $a,b,c$ we have
\begin{enumerate}[(a)]
\item $\chi_3([a,b,c])\leq 6\chi_1([a,b,c])$ 
\item $\chi_3([a,b,c])\leq 3\chi_1([a,b,c])$  when $|\{a,b,c\}|=2$
\item $\chi_3([a,b,c])\leq 3\chi_2([a,b,c])$ 
\item $\chi_2([a,b,c])\leq 2\chi_1([a,b,c])$. 
\end{enumerate}
\end{lemma}

%

\subsection{Maximal minimal valency}\label{maxmin}

We define $\mv_\bullet([a,b,c])$ as the maximal minimal valency of a graph in $\CG_\bullet([a,b,c])$, also known as  the \emph{degeneracy} of the class  $\CG_\bullet([a,b,c])$.
The standard argument shows that $\chi_\bullet([a,b,c])\leq  \mv_\bullet([a,b,c])+1$. In most cases, we have better bounds avaliable, but for some, this remains our most efficient approach.

 Since all configurations studied here are finite, we can in turn bound $\mv_\bullet([a,b,c])$ by 
noting that every configuration in $\CG_\bullet([a,b,c])$ must have at least one cuboid which has no neighbor below it. Arguing on rectangles in the plane representing all cuboids intersecting the lowest open thickened plane $\RR\times \RR\times (z,z+1)$ we further get that at least one such rectangle has two free bottom and right  faces, so that the cuboid has 
no neighbor touching the lower band of height 1 on its front and right sides, as illustrated below.
\begin{center}
\includegraphics[width=3.5cm]{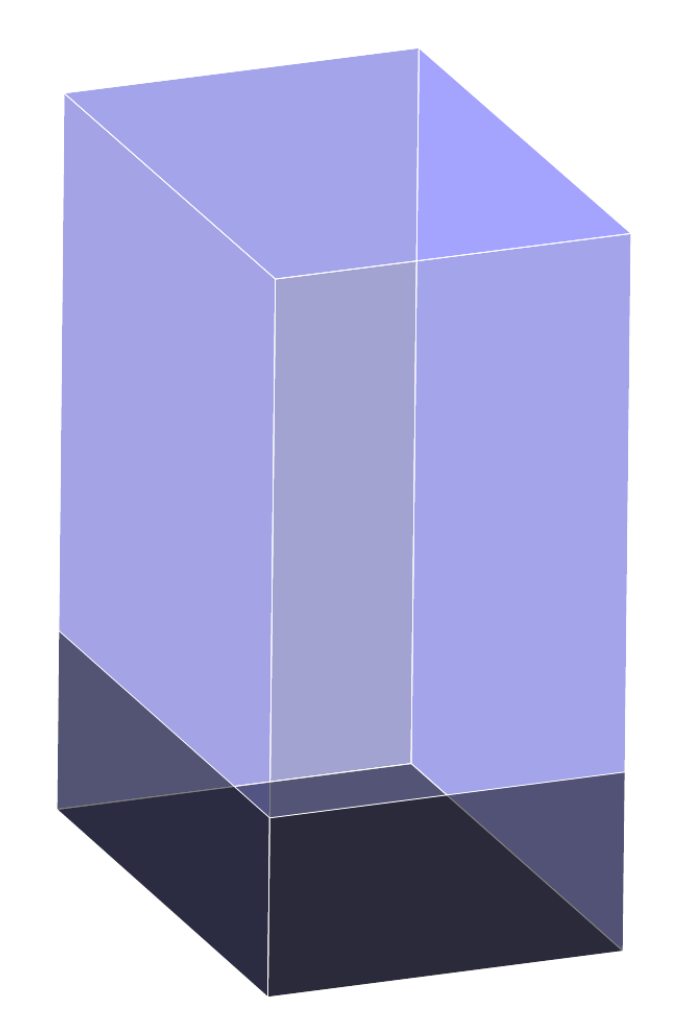}
\end{center}
 Letting $\nv_\bullet([a,b,c])$ denote the largest number of such neighbors   a cuboid in $\myC_\bullet([a,b,c])$ can have, we conclude
 \[
\mv_\bullet([a,b,c])\leq \nv_\bullet([a,b,c])
 \]
 and note that $\nv_\bullet([a,b,c])$ can be computed as an independence number. To do so, we fix a cuboid  $c_0=[0,a']\times [0,b']\times [0,c']$ with $(a',b',c')$ one of the permutations of $(a,b,c)$ allowed in the definition of $\CG_\bullet([a,b,c])$ and add a vertex for each element $c\in \myC_\bullet([a,b,c])$ which neighbors 
 $c_0$ and such that the root $(x,y,z)$ of $c$ satisfies
 \begin{equation}\label{notouch}
 x>0\vee (x=0\wedge (y>0\wedge z>0))
 \end{equation}
Two vertices are then connected by an edge if the cuboids collide, and the independence number is then exactly the maximal number of non-colliding neighbors. The number of neighbors varies with the position of the center cuboid because of the non-symmetric formulation of \eqref{notouch}. For instance, as illustrated in
 \begin{center}
\includegraphics[width=5cm]{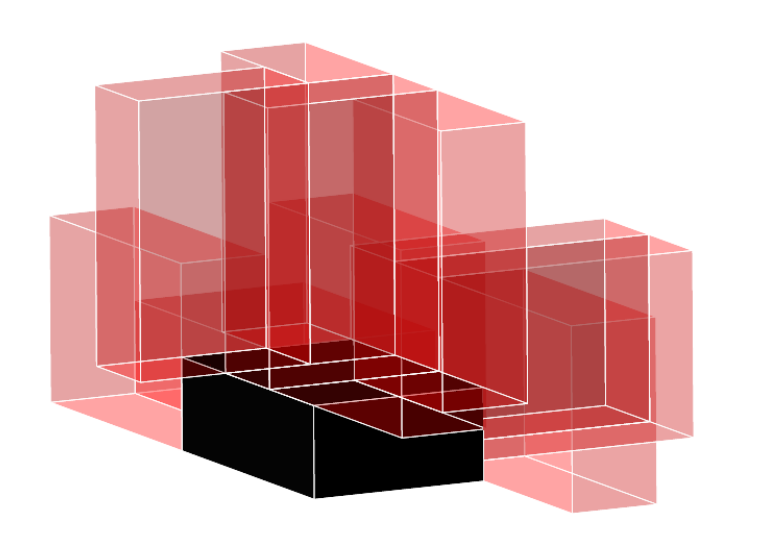}\includegraphics[width=5cm]{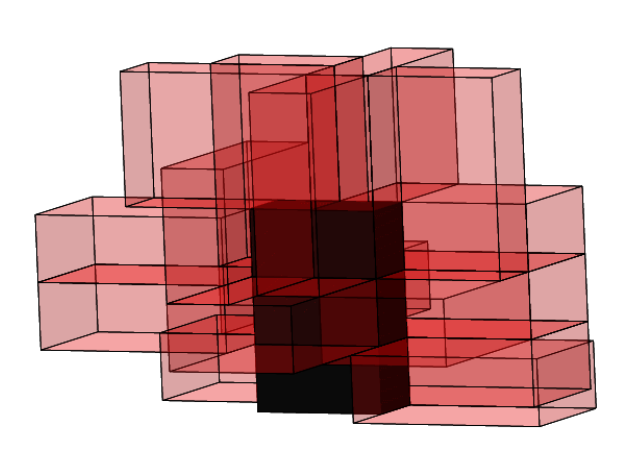}
\end{center}
the maximal number of neighbors of $[0,1]\times[0,2]\times[0,3]$ to the left is 10, whereas it is 16 for $[0,3]\times[0,2]\times[0,1]$. It suffices to compute the graphs and independence numbers for three of the six positions using mirror symmetry, and then one uses the largest such number as $\nv$.


\subsection{Pseudorandom searches}\label{pseudo}

As detailed in the introduction, almost all lower bounds we present in this paper are computer-generated. 
But configurations having high chromatic numbers are rare in the sense, e.g., that if one fixes a cube $M\times M\times M$ and fills it up randomly (allowing for holes appropriately), then the probability of finding such a configuration is small. Since one does not know, a priori, whether a configuration with a given high chromatic number can exist in such a given box, more efficient methods for finding such examples had to be developed.

This was a theme in the Experimental Mathematics course at University of Copenhagen in 2012--14,  where many approaches were tried and some found to be succesful in the very palpable sense of locating examples improving the previous lower bound for the numbers $\chi_2([a,b,1])$. From many approaches tested, the  following algorithm appeared  to give the most interesting results:

\begin{algorithm}
	\caption{}
	\begin{algorithmic}[1]
	\State Place $n_{00}$ non-touching cuboids in the box $M\times M\times M$
	\State Color all cuboids with color 1
	\State $\chi\leftarrow 1, n\leftarrow n_{00}$
		\While {$\chi<\chi_0$ and $n<n_0$}
			\State Find all positions having the maximal number of neighboring colors
			\State Add one such cuboid at random to the configuration, $n\leftarrow n+1$
			\State Recolor the configuration and update $\chi\leftarrow \chi+1$ if necessary
		\EndWhile
		\If {$\chi=\chi_0$}
		\State \Return A critical subconfiguration
		\Else
		\State \Return $\emptyset$
		\EndIf
	\end{algorithmic} 
\end{algorithm}
Is it not unreasonable to think of the algorithm as executing a game where one side (in line 6) tries to bring the chromatic number upwards, and the other side (in line 7) tries to keep it down. 
Note for line 10 that it is straightforward to obtain a configuration which is minimal in size (i.e. critical) by removing all cuboids in turn and computing the chromatic number of the remaining configuration. If it remains the same, the cuboid can be left out.
But one notes also that the recomputation of chromatic numbers is computationally costly, and in practice we were not able to find critical examples with much more than 25 cuboids in it this way.

For our more recent and more systematic attempt to search for configurations establishing further lower bounds, we first tried to apply the same algorithm but employing technological advances. We used our computers (standard laptops) such as they were, but migrated from computing chromatic numbers in a computer algebra system to using a satisfiability solver, viz. the Python software package called PySat \cite{pysat}. As is well known, determining whether a graph $ G $ can be colored with $ n$ colors can be formulated as a Boolean satisfiability problem as follows: for each node $ v \in G $, we define $ n $ Boolean variables $ v_1, \ldots, v_n$, where $v_i$ is true if node $v$ is assigned color $i$. The satisfiability formula consists of clauses  $(v_1 \vee \ldots \vee v_n)$ for each $v \in G$, ensuring that each node receives a color. Additionally, for each edge $(v, w) \in G$ , we include clauses $\neg (v_i \wedge w_i)$ for $i = 1, \ldots, n $, ensuring that adjacent nodes are not assigned the same color.
This process can be repeated for increasing values of $n$ until the above formula is satisfiable, thereby providing a proper coloring for the graph.
We also used the Python package NetworkX \cite{networkx} to store contact graphs of cuboid configurations and to extract information about various graph attributes, such as the valency of specific vertices. 

This led to significant advances in computation times, but not to many new examples. We consequently tried other approaches, and found the following to be superior:

\begin{algorithm}
	\caption{}
	\begin{algorithmic}[1]
	\State Place $n_{00}$ non-touching cuboids in the box $M\times M\times M$
	\State $n\leftarrow n_{00}$
		\While { $n<n_0$}
			\State Find all positions having the maximal number of neighbors
			\State Add one such cuboid at random to the configuration, $n\leftarrow n+1$
		\EndWhile
		\State Color the configuration and compute $\chi$
		\If {$\chi=\chi_0$}
		\State \Return A critical subconfiguration
		\Else
		\State \Return $\emptyset$
		\EndIf
	\end{algorithmic} 
\end{algorithm}

Obviously this approach is much faster because we have eliminated the need to continually compute chromatic numbers, and this in turn allows the search for much larger configurations.
We  expect the  original approach of placing cuboids based on already found colors to reach high chromatic numbers  more efficiently  than this approach of just aiming for high valencies, but whatever disadvantage is very obviously made up for in increased options for experimentation in moderate time. This version has led us to  all critical examples presented with 30 or more cuboids.

\subsection{Periodic colorings}\label{periodic}

An \emph{inherited} $n$-coloring of elements in $\myC_\bullet([a,b,c])$ in its most general form is a map $\kappa:\ZZ^3\times S_\bullet\to\{1,\dots,n\}$ ($S_\bullet$ being the symmetric group of relevant size) which assigns the color $\kappa(x,y,z,\pi)$ to a cuboid rooted at $(x,y,z)$ and with the dimensions given by the given permutation of the tuple $(a,b,c)$. We call the coloring \emph{periodic} if for some $(x_0,y_0,z_0)\in\NN^3$
\[
\kappa(x+x_0,y+y_0,z+z_0,\pi)=\kappa(x,y,z,\pi).
\]
When $\bullet=1$, there is no choice of direction, and we tacitly leave out the entry $\pi$. When $\bullet=2$, we write $S_2=\{\id,\tau\}$.

We use the notation $\perco{\bullet}{a}{b}{c}{x}{y}{z}$ to indicate the smallest number of colors that can be used to color any element of $\CG_\bullet([a,b,c])$ using a periodic coloring that is periodic with period $x\times y\times z$, or $\infty$ if no such coloring exists. 

This number can in principle be computed by defining a graph with vertices that are all elements in $\myC_\bullet([a,b,c])$ with a root in $[0,x)\times [0,y)\times [0,z)$ and edges whenever two such cuboids, or any of their periodic translates, do not overlap but do touch. If the graph thus defined has loops, $\perco{\bullet}{a}{b}{c}{x}{y}{z}=\infty$, if not, $\perco{\bullet}{a}{b}{c}{x}{y}{z}
$ is its chromatic number. But such graphs appear to form challenges for our computational approaches, and computing their chromatic numbers is often outside of reach, even for graphs as in Section \ref{pseudo} of the same or larger size. 

For all the examples we have found, we are able to give geometric explanations and proofs of the upper bounds found this way, and hence we do not need to rely on computer-based computations in these cases to verify the claims.

\subsection{Clique numbers}

We end this section by noting that the clique numbers for graphs in $\CG_\bullet([a,b,c])$ can never exceed 4. It is only possible to realize $\chi_\bullet([a,b,c])$ by cliques in very few cases, and indeed this bound on clique numbers should  be interpreted as an explanation of the fact that we must often resort to rather large configurations to realize our best lower bounds of these values.

Since the argument is the same in all dimensions, we argue for generalized axis-parallel hypercuboids in $\RR^d$, with the convention that two hypercuboids \textbf{touch} if their interiors are disjoint, but their intersection is a hypercuboid of dimension $d-1$. This is consistent with our previous definition at $d=3$, where cuboids must intersect in non-degenerate rectangles. We prove:

\begin{theor}\label{bdclique}
Any contact graph describing touching hypercuboids of the form
\[
[x_1,x_1+a_1]\times \cdots \times [x_d,x_d+a_d]
\]
in $\RR^d$ has clique number at most $d+1$.
\end{theor}

\begin{lemma}\label{nonempty} 
If $C_1,\dots,C_n$ are hypercuboids in $\RR^d$ of the form
\[
[x_1,x_1+a_1]\times \cdots \times [x_d,x_d+a_d]
\]
that all pairwise intersect, then 
\[
\bigcap_{i=1}^n C_i\not=\emptyset
\]
\end{lemma}
\begin{proof}
We note that each $C_i$ is the intersection of $2^d$ subsets of the form
\[
\RR\times\RR\times [x_k,\infty)\times \RR\times \cdots \times\RR
\]
or 
\[
\RR\times\RR\times (-\infty,y_k]\times \RR\times \cdots \times\RR
\]
Replacing each hypercuboid by these $2^d$ half-spaces, we obtain $n2^d$ such sets that all pairwise overlap. This implies that any $x_k$ thus appearing must be dominated by any $y_k$, and consequently at least one element of $\RR^d$ is contained in all half-spaces.
\end{proof}

\begin{proofof}{Theorem \ref{bdclique}}
Let  $C_1,\dots,C_n$ be a family of hypercuboids that all pairwise touch; we need to show that $n\leq d+1$. By Lemma \ref{nonempty} there is a $\mathbf v\in \RR^d$ which is contained in all $C_i$. Since no interior point of $C_i$ is contained in any $C_j$ with $j \neq i$, it follows that $\mathbf v$ is on the boundary of each $C_i$.

%

We may assume after translation that $\mathbf v=\mathbf 0$ and proceed to color the {edges} of the complete graph $K_n$ as follows. When $i\not= j$, we know that $C_i\cap C_j$ contains a $(d-1)$-dimensional hypercuboid which is perpendicular to a unique standard basis vector $\mathbf e_k$, and color the edge $(i,j)\in K_n$ by the color $k$. 
We note that every edge of $K_n$ gets a color and also that the edges of color $k$ form a bipartite graph where one of the bipartite classes consists of all the hypercuboids $C_i$ in the positive halfspace of the hyperplane perpendicular to
$\mathbf e_k$, and the other bipartite class consists of all hypercuboids $C_i$ in the negative halfspace. Note that this bipartite graph is indeed complete bipartite, that is, all edges between the two bipartite classes are present.

Since this holds true for all $k\in\{1,\dots, d\}$, we have produced a decomposition of $K_n$ into $d$ complete bipartite graphs. By \cite{tverberg}, this implies $d\geq n-1$ as desired.
\end{proofof}

There are  graphs with clique number 4 in all $\CG_\bullet([a,b,c])$ studied here except $\CG_\bullet([1,1,1])$ and $\CG_1([a,1,1])$.

\section{$\chi_1$}

We first present  examples establishing lower bounds for $\chi_1$. These examples were all found by selecting minimal examples realizing the given chromatic number  amongst a large body of experiments generated pseudo-randomly, cf. Section \ref{pseudo}. Since we have proved in Lemma \ref{increases} that $\chi_1([a,b,c])$ is increasing in the dimensions, the results given below imply lower bounds as presented at all larger cuboid sizes.

\begin{propo}\label{lowerbounds}
\begin{gather*}
\chi_1([3,1,1])\geq 4\\
\chi_1([2,2,1])\geq 5\\
\chi_1([5,2,1])\geq 6\\
\chi_1([4,3,1])\geq 6\\
\chi_1([2,2,2])\geq 6.
\end{gather*}
\end{propo}
\begin{proof}
The first claim is seen by the 10-block building
\begin{center}
\includegraphics[width=6cm]{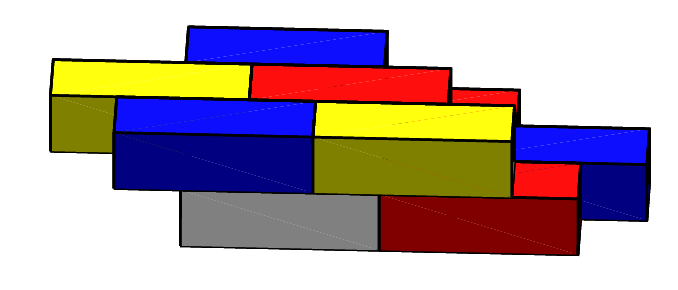}
\end{center}
requiring 4 colors, which in exploded view is
\begin{center}
\includegraphics[width=6cm]{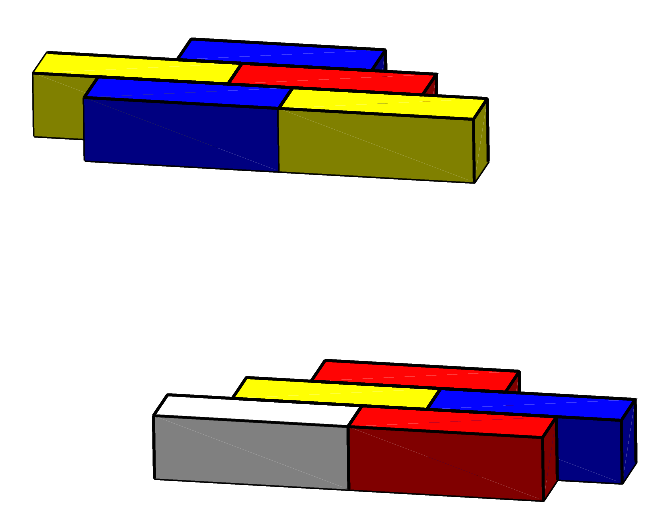}
\end{center}
For $2\times 2\times 1$ we have the 11-block building
\begin{center}
\includegraphics[width=6cm]{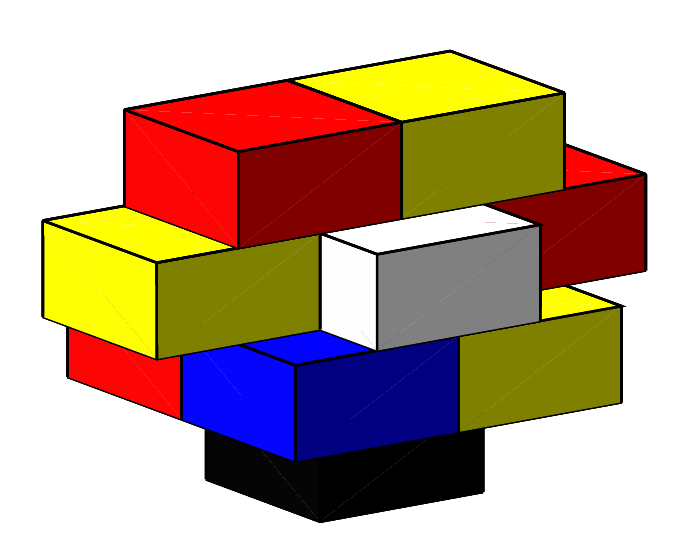}
\end{center}
requiring 5 colors, which in exploded view is
\begin{center}
\includegraphics[width=4cm]{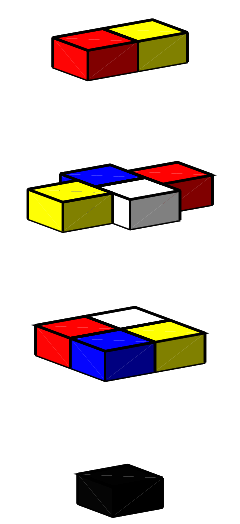}
\end{center}
It is fairly easy to check these claims by hand. In the first, one rules out a proper 3-coloring by noting that both layers have a unique 3-coloring up to permutation of the colors, and  then check that no choice of permutations is coherent across the layers. In the second example, one notes that if one had only four colors, the two top layers would have to be colored as indicated due to the central 4-clique. One can then show that the third layer cannot have any repeated colors, and consequently a fifth color is needed at the bottom.

The remaining buildings all require 6 colors and are difficult to check by hand. For $5\times 2\times 1$ we have the 190-cuboid building
which in exploded view is
\begin{center}
\includegraphics[width=8cm]{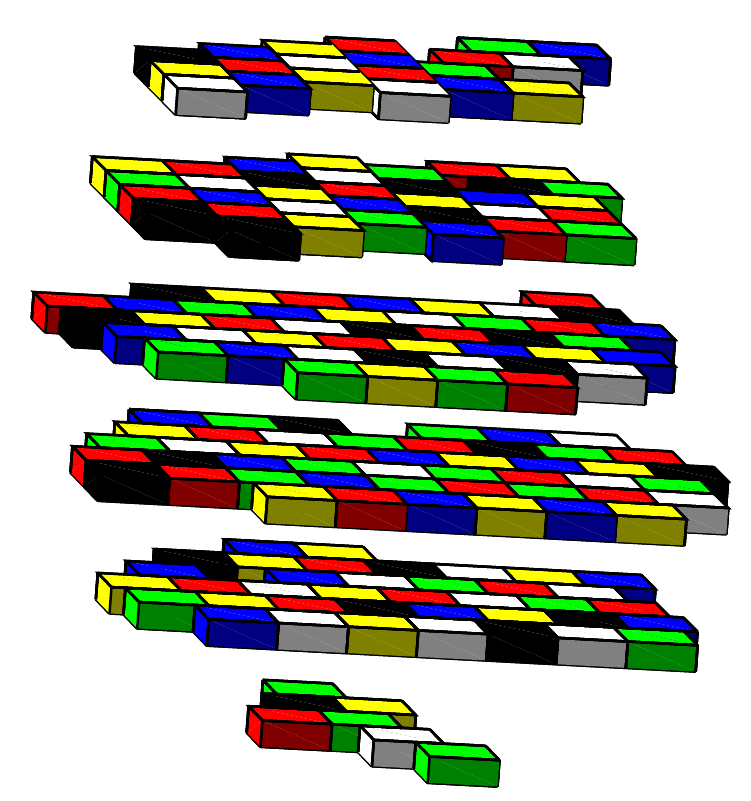}
\end{center}
For $4\times 3\times 1$ we have the perhaps more understandable 44-cuboid building
\begin{center}
\includegraphics[width=6cm]{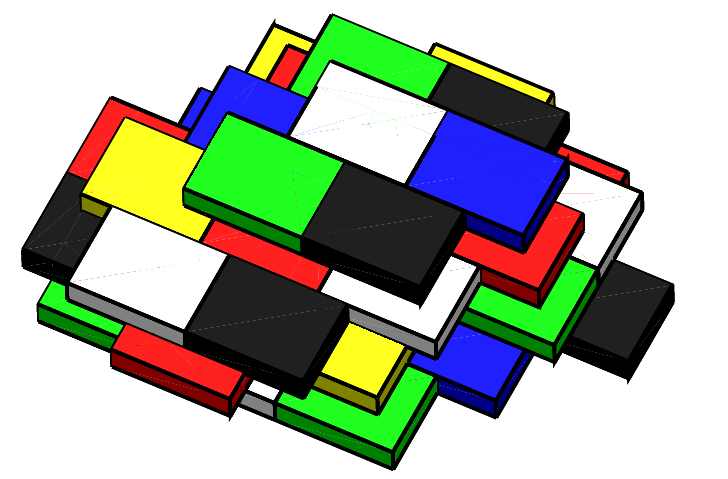}
\end{center}
which in exploded view is
\begin{center}
\includegraphics[width=8cm]{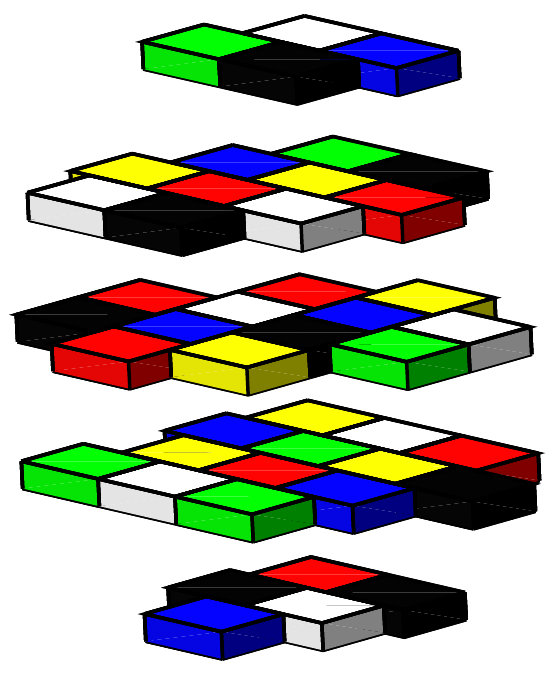}
\end{center}
and finally for $2\times 2\times 2$ we have the 75-cuboid building
\begin{center}
\includegraphics[width=8cm]{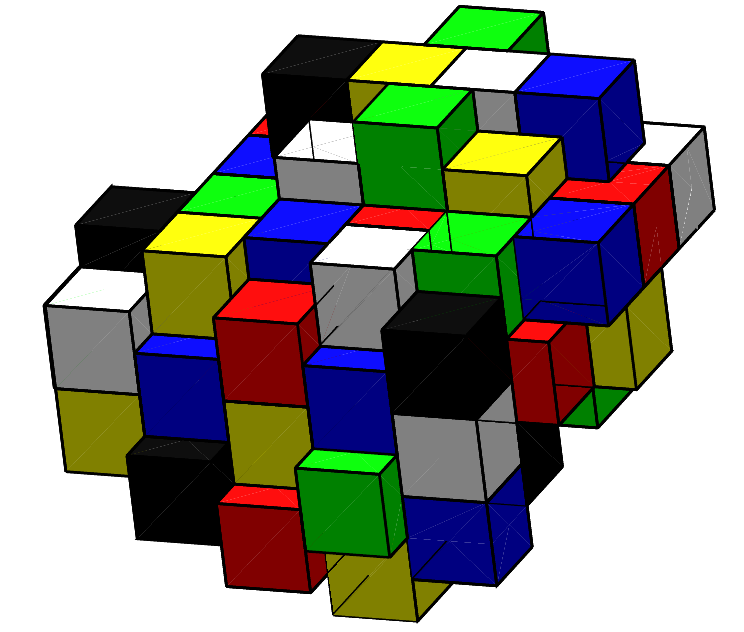}
\end{center}
which is not layered and hence not usefully displayable in exploded view.
\end{proof}

\begin{remar}
We recall from Section \ref{pseudo} that all of the previous configurations are critical -- if any one cuboid is removed, the chromatic number decreases by one.

The first example of a building with $2\times 2\times1$ requiring 5 colors was found in 2012 by Nicolas Bru Frantzen,
Mikkel B\o{}hlers Nielsen and 
Alex Voigt Hansen  as students in the University of Copenhagen course ``Experimental mathematics''. The one presented here has slightly fewer blocks.
\end{remar}

We now move on to upper bounds.

\begin{theor}\label{boundi}\mbox{}
\begin{enumerate}[(a)]
\item $\chi_1([1,1,1])\leq \perco{1}{1}{1}{1}{2}{2}{2}=2$.
\item $\chi_1([2,1,1])\leq \perco{1}{2}{1}{1}{6}{2}{2}=3$.
\item $\chi_1([a,1,1])\leq \perco{1}{a}{1}{1}{2a}{2}{2}= 4$ for all $a\geq 3$.
\item  $\chi_1([2,2,1])\leq \perco{1}{2}{2}{1}{10}{10}{2}= 5$.
\item $\chi_1([a,2,1])\leq \perco{1}{a}{2}{1}{2a}{6}{2}\leq 6$ for all $a\geq 3$.
\item $\chi_1([a,3,1])\leq \perco{1}{a}{3}{1}{2a}{21}{2}\leq 7$ for all $a\geq 3$.
\item $\chi_1([a,4,1])\leq \perco{1}{a}{4}{1}{2a}{28}{2}\leq 7$ for all $a\geq 4$.
\item $\chi_1([a,b,c])\leq\perco{1}{a}{b}{c}{2a}{2b}{2c}\leq 8$ for all $a\geq b\geq c$.
\end{enumerate}
\end{theor}
\begin{proof}
We indicate colorings diagrammatically. All these colorings will be so completely systematic in the concrete choice of colors that it suffices to check the color 1 in these respects.
In our diagrams, the shading shows those regions that might be occupied by a block with the color 1; we must check that these regions do not touch. Also, we must check that the repeated colors in rectangular configurations are close enough that two cuboids placed there would collide, and hence can be colored with the same color.

The first observation is just the standard checkerboard coloring in two alternating layers:
\begin{center}
\begin{tabular}{|c|c|}\hline
\nsp{1}&2\\\hline
2&\nsp{1}\\\hline
\end{tabular}\qquad
\begin{tabular}{|c|c|}\hline
2&\nsp 1\\\hline
\nsp 1&2\\\hline
\end{tabular}
\end{center}

For $2\times 1\times 1$, a periodic coloring is given by  
\begin{center}
\begin{tabular}{|c|c|c|c|c|c|}\hline
\nsp{1}&\nsp{1}&\nsp{2}&2&3&3\\\hline
2&3&3&\nsp{1}&\nsp{1}&\nsp{2}\\\hline
\end{tabular}\qquad
\begin{tabular}{|c|c|c|c|c|c|}\hline
2&3&3&\nsp 1&\nsp 1&\nsp 2\\\hline
\nsp 1&\nsp 1&\nsp 2&2&3&3\\\hline
\end{tabular}
\end{center}

For $a\times 1\times 1$, illustrated here for $a=4$, we use 
\begin{center}
\begin{tabular}{|c|c|c|c|c|c|c|c|c|c|c|c|c|c|c|c|}\hline
\nsp 1&\nsp 1&\nsp 1&\nsp 1&\nsp 2&\nsp 2&\nsp 2&2\\\hline
3&3&3&3&4&4&4&4\\\hline
\end{tabular}
\end{center}
and
\begin{center}
\begin{tabular}{|c|c|c|c|c|c|c|c|c|c|c|c|c|c|c|c|}\hline
3&3&3&3&4&4&4&4\\\hline
\nsp 1&\nsp 1&\nsp 1&\nsp 1&\nsp 2&\nsp 2&\nsp 2&2\\\hline
\end{tabular}
\end{center}

For $2\times 2\times 1$ we have
\begin{center}
\begin{tabular}{|c|c|c|c|c|c|c|c|c|c|}\hline
\nsp 1&\nsp 1&\nsp 2&2&3&3&4&4&5&5\\\hline
\nsp 1&\nsp 1&\nsp 2&2&\nsp 3&\nsp 3&\nsp 4&4&5&5\\\hline
4&4&5&5&\nsp 1&\nsp 1&\nsp 2&2&3&3\\\hline
\nsp 4&4&5&5&\nsp 1&\nsp 1&\nsp 2&2&\nsp 3&\nsp3\\\hline
\nsp 2&2&3&3&4&4&5&5&\nsp 1&\nsp 1\\\hline
\nsp 2&2&\nsp 3&\nsp 3&\nsp 4&4&5&5&\nsp 1&\nsp 1\\\hline
5&5&\nsp 1&\nsp 1&\nsp 2&2&3&3&4&4\\\hline
5&5&\nsp 1&\nsp 1&\nsp 2&2&\nsp 3&\nsp 3&\nsp 4&4\\\hline
3&3&4&4&5&5&\nsp 1&\nsp 1&\nsp 2&2\\\hline
\nsp 3&\nsp 3&\nsp 4&4&5&5&\nsp 1&\nsp 1&\nsp 2&2\\\hline
\end{tabular}
\end{center}
and
\begin{center}
\begin{tabular}{|c|c|c|c|c|c|c|c|c|c|}\hline
4&5&5&\nsp 1&\nsp 1&\nsp 2&2&\nsp 3&\nsp3&\nsp 4\\\hline
2&3&3&4&4&5&5&\nsp 1&\nsp 1&\nsp 2\\\hline
2&\nsp 3&\nsp 3&\nsp 4&4&5&5&\nsp 1&\nsp 1&\nsp 2\\\hline
5&\nsp 1&\nsp 1&\nsp 2&2&3&3&4&4&5\\\hline
5&\nsp 1&\nsp 1&\nsp 2&2&\nsp 3&\nsp 3&\nsp 4&4&5\\\hline
3&4&4&5&5&\nsp 1&\nsp 1&\nsp 2&2&3\\\hline
\nsp 3&\nsp 4&4&5&5&\nsp 1&\nsp 1&\nsp 2&2&\nsp 3\\\hline
\nsp 1&\nsp 2&2&3&3&4&4&5&5&\nsp 1\\\hline
\nsp 1&\nsp 2&2&\nsp 3&\nsp 3&\nsp 4&4&5&5&\nsp 1\\\hline
4&5&5&\nsp 1&\nsp 1&\nsp 2&2&3&3&4\\\hline
\end{tabular}
\end{center}

The colorings (d)--(f) are similar in nature; we present them for the square case $a\times a\times 1$ and with the $x$ and $y$ coordinates transposed to conserve space. It is straightforward to extend to $b\times a\times 1$ with $b>a$ by just expanding in the vertical (short) direction. For $2\times 2\times 1$, we use
\begin{center}
\begin{tabular}{|c|c|c|c|c|c|c|c|c|c|c|c|}\hline
\nsp 1&\nsp 1&\nsp 2&2&3&3& 4& 4& 5&5&6&6\\\hline
\nsp 1&\nsp 1&\nsp 2&2&3&3&\nsp 4&\nsp 4&\nsp 5&5&6&6\\\hline
4&4&5&5&6&6&\nsp 1&\nsp 1&\nsp 2&2&3&3\\\hline
\nsp 4&\nsp 4&\nsp 5&5&6&6&\nsp 1&\nsp 1&\nsp 2&2&3&3\\\hline
\end{tabular}
\end{center}
and
\begin{center}
\begin{tabular}{|c|c|c|c|c|c|c|c|c|c|c|c|}\hline
5&6&6&\nsp 1&\nsp 1&\nsp 2&2&3&3&4& 4& 5\\\hline
5&6&6&\nsp 1&\nsp 1&\nsp 2&2&3&3&\nsp 4&\nsp 4&\nsp 5\\\hline
2&3&3&4&4&5&5&6&6&\nsp 1&\nsp 1&\nsp 2\\\hline
2&3&3&\nsp 4&\nsp 4&\nsp 5&5&6&6&\nsp 1&\nsp 1&\nsp 2\\\hline
\end{tabular}
\end{center}
This uses one more color than what we need for $2\times 2\times 1$, but generalizes to all $a\times 2\times 1$ as opposed to the $10\times 10\times 2$ coloring presented above.

For $3\times 3\times 1$, we have
\begin{center}
\begin{tabular}{|c|c|c|c|c|c|c|c|c|c|c|c|c|c|c|c|c|c|c|c|c|}\hline
\nsp 1&\nsp 1&\nsp1&\nsp 2&\nsp 2&2&3&3&3& 4& 4& 4&5&5&5&6&6&6&7&7&7\\\hline
\nsp 1&\nsp 1&\nsp1&\nsp 2&\nsp 2&2&3&3&3& 4& \nsp 4& \nsp 4&\nsp 5&\nsp 5&\nsp 5&6&6&6&7&7&7\\\hline
\nsp 1&\nsp 1&\nsp1&\nsp 2&\nsp 2&2&3&3&3& 4& \nsp 4& \nsp 4&\nsp 5&\nsp 5&\nsp 5&6&6&6&7&7&7\\\hline
  4&5&5&5&6&6&6&7&7&7&\nsp 1&\nsp 1&\nsp1&\nsp 2&\nsp 2&2&3&3&3& 4&4\\\hline
  \nsp 4&\nsp 5&\nsp 5&\nsp 5&\nsp 6&6&6&7&7&7&\nsp 1&\nsp 1&\nsp1&\nsp 2&\nsp 2&2&3&3&3& 4&4\\\hline
  \nsp 4&\nsp 5&\nsp 5&\nsp 5&\nsp 6&6&6&7&7&7&\nsp 1&\nsp 1&\nsp1&\nsp 2&\nsp 2&2&3&3&3& 4&4\\\hline
    \end{tabular}
\end{center}
and
\begin{center}
\begin{tabular}{|c|c|c|c|c|c|c|c|c|c|c|c|c|c|c|c|c|c|c|c|c|}\hline
6&6&7&7&7&\nsp 1&\nsp 1&\nsp1&\nsp 2&\nsp 2&2&3&3&3& 4& 4& 4&5&5&5&6\\\hline
6&6&7&7&7&\nsp 1&\nsp 1&\nsp1&\nsp 2&\nsp 2&2&3&3&3& 4& \nsp 4& \nsp 4&\nsp 5&\nsp 5&\nsp 5&6\\\hline
6&6&7&7&7&\nsp 1&\nsp 1&\nsp1&\nsp 2&\nsp 2&2&3&3&3& 4& \nsp 4& \nsp 4&\nsp 5&\nsp 5&\nsp 5&6\\\hline
3&3&3& 4&4&  4&5&5&5&6&6&6&7&7&7&\nsp 1&\nsp 1&\nsp1&\nsp 2&\nsp 2&2\\\hline
3&3&3& 4&4&\nsp  4&\nsp 5&\nsp 5&\nsp 5&\nsp 6&6&6&7&7&7&\nsp 1&\nsp 1&\nsp1&\nsp 2&\nsp 2&2\\\hline
3&3&3& 4&4&  \nsp 4&\nsp 5&\nsp 5&\nsp 5&\nsp 6&6&6&7&7&7&\nsp 1&\nsp 1&\nsp1&\nsp 2&\nsp 2&2\\\hline
    \end{tabular}
\end{center}
\newcommand{\sss}[1]{\!{\scriptsize #1}\!}
\newcommand{\nss}[1]{\nsp{\sss{#1}}}
and for $4\times 4\times 1$
\begin{center}
\begin{tabular}{|c|c|c|c|c|c|c|c|c|c|c|c|c|c|c|c|c|c|c|c|c|c|c|c|c|c|c|c|}\hline
\nss 1&\nss 1&\nss1&\nss 1& \nss 2&\nss 2&\nss 2&\sss 2&\sss 3&\sss 3&\sss 3&\sss 3&\sss 4&\sss 4&\sss 4&\sss 4&\sss 5&\sss 5&\sss 5&\sss 5&\sss 6&\sss 6&\sss 6&\sss 6&\sss 7&\sss 7&\sss 7&\sss 7\\\hline
\nss 1&\nss 1&\nss1&\nss 1& \nss 2&\nss 2&\nss 2&\sss 2&\sss 3&\sss 3&\sss 3&\sss 3&\sss 4&\sss 4&\nss 4&\nss 4&\nss 5&\nss 5&\nss 5&\nss 5&\nss 6&\sss 6&\sss 6&\sss 6&\sss 7&\sss 7&\sss 7&\sss 7\\\hline
\nss 1&\nss 1&\nss1&\nss 1& \nss 2&\nss 2&\nss 2&\sss 2&\sss 3&\sss 3&\sss 3&\sss 3&\sss 4&\sss 4&\nss 4&\nss 4&\nss 5&\nss 5&\nss 5&\nss 5&\nss 6&\sss 6&\sss 6&\sss 6&\sss 7&\sss 7&\sss 7&\sss 7\\\hline
\nss 1&\nss 1&\nss1&\nss 1& \nss 2&\nss 2&\nss 2&\sss 2&\sss 3&\sss 3&\sss 3&\sss 3&\sss 4&\sss 4&\nss 4&\nss 4&\nss 5&\nss 5&\nss 5&\nss 5&\nss 6&\sss 6&\sss 6&\sss 6&\sss 7&\sss 7&\sss 7&\sss 7\\\hline
\sss 4&\sss 4&\sss 5&\sss 5&\sss 5&\sss 5&\sss 6&\sss 6&\sss 6&\sss 6&\sss 7&\sss 7&\sss 7&\sss 7&\nss 1&\nss 1&\nss1&\nss 1& \nss 2&\nss 2&\nss 2&\sss 2&\sss 3&\sss 3&\sss 3&\sss 3&\sss 4&\sss 4\\\hline
\nss 4&\nss 4&\nss 5&\nss 5&\nss 5&\nss 5&\nss 6&\sss 6&\sss 6&\sss 6&\sss 7&\sss 7&\sss 7&\sss 7&\nss 1&\nss 1&\nss1&\nss 1& \nss 2&\nss 2&\nss 2&\sss 2&\sss 3&\sss 3&\sss 3&\sss 3&\sss 4&\sss 4\\\hline
\nss 4&\nss 4&\nss 5&\nss 5&\nss 5&\nss 5&\nss 6&\sss 6&\sss 6&\sss 6&\sss 7&\sss 7&\sss 7&\sss 7&\nss 1&\nss 1&\nss1&\nss 1& \nss 2&\nss 2&\nss 2&\sss 2&\sss 3&\sss 3&\sss 3&\sss 3&\sss 4&\sss 4\\\hline
\nss 4&\nss 4&\nss 5&\nss 5&\nss 5&\nss 5&\nss 6&\sss 6&\sss 6&\sss 6&\sss 7&\sss 7&\sss 7&\sss 7&\nss 1&\nss 1&\nss1&\nss 1& \nss 2&\nss 2&\nss 2&\sss 2&\sss 3&\sss 3&\sss 3&\sss 3&\sss 4&\sss 4\\\hline
    \end{tabular}
\end{center}
along with 
\begin{center}
\begin{tabular}{|c|c|c|c|c|c|c|c|c|c|c|c|c|c|c|c|c|c|c|c|c|c|c|c|c|c|c|c|}\hline
\sss 6&\sss 6&\sss 6&\sss 7&\sss 7&\sss 7&\sss 7&\nss 1&\nss 1&\nss1&\nss 1& \nss 2&\nss 2&\nss 2&\sss 2&\sss 3&\sss 3&\sss 3&\sss 3&\sss 4&\sss 4&\sss 4&\sss 4&\sss 5&\sss 5&\sss 5&\sss 5&\sss 6\\\hline
\sss 6&\sss 6&\sss 6&\sss 7&\sss 7&\sss 7&\sss 7&\nss 1&\nss 1&\nss1&\nss 1& \nss 2&\nss 2&\nss 2&\sss 2&\sss 3&\sss 3&\sss 3&\sss 3&\sss 4&\sss 4&\nss 4&\nss 4&\nss 5&\nss 5&\nss 5&\nss 5&\nss 6\\\hline
\sss 6&\sss 6&\sss 6&\sss 7&\sss 7&\sss 7&\sss 7&\nss 1&\nss 1&\nss1&\nss 1& \nss 2&\nss 2&\nss 2&\sss 2&\sss 3&\sss 3&\sss 3&\sss 3&\sss 4&\sss 4&\nss 4&\nss 4&\nss 5&\nss 5&\nss 5&\nss 5&\nss 6\\\hline
\sss 6&\sss 6&\sss 6&\sss 7&\sss 7&\sss 7&\sss 7&\nss 1&\nss 1&\nss1&\nss 1& \nss 2&\nss 2&\nss 2&\sss 2&\sss 3&\sss 3&\sss 3&\sss 3&\sss 4&\sss 4&\nss 4&\nss 4&\nss 5&\nss 5&\nss 5&\nss 5&\nss 6\\\hline
\sss 2&\sss 3&\sss 3&\sss 3&\sss 3&\sss 4&\sss 4\sss 4&\sss 4&\sss4&\sss 5&\sss 5&\sss 5&\sss 5&\sss 6&\sss 6&\sss 6&\sss 6&\sss 7&\sss 7&\sss 7&\sss 7&\nss 1&\nss 1&\nss1&\nss 1& \nss 2&\nss 2&\nss 2\\\hline
\sss 2&\sss 3&\sss 3&\sss 3&\sss 3&\sss 4&\sss 4&\nss 4&\nss 4&\nss 5&\nss 5&\nss 5&\nss 5&\nss 6&\sss 6&\sss 6&\sss 6&\sss 7&\sss 7&\sss 7&\sss 7&\nss 1&\nss 1&\nss1&\nss 1& \nss 2&\nss 2&\nss 2\\\hline
\sss 2&\sss 3&\sss 3&\sss 3&\sss 3&\sss 4&\sss 4&\nss 4&\nss 4&\nss 5&\nss 5&\nss 5&\nss 5&\nss 6&\sss 6&\sss 6&\sss 6&\sss 7&\sss 7&\sss 7&\sss 7&\nss 1&\nss 1&\nss1&\nss 1& \nss 2&\nss 2&\nss 2\\\hline
\sss 2&\sss 3&\sss 3&\sss 3&\sss 3&\sss 4&\sss 4&\nss 4&\nss 4&\nss 5&\nss 5&\nss 5&\nss 5&\nss 6&\sss 6&\sss 6&\sss 6&\sss 7&\sss 7&\sss 7&\sss 7&\nss 1&\nss 1&\nss1&\nss 1& \nss 2&\nss 2&\nss 2\\\hline    \end{tabular}
\end{center}

In the last and general case, the idea is simply to color each octant of the $2a\times 2b\times 2c$ region with a separate color. For $2\times 2\times 2$ we would use
\begin{center}
\begin{tabular}{|c|c|c|c|}\hline
\nsp 1&\nsp 1&\nsp 2& 2\\\hline
\nsp 1&\nsp 1&\nsp 2& 2\\\hline
 3&3& 4&4\\\hline
\nsp 3&\nsp 3&\nsp 4&4\\\hline
\end{tabular}\quad
\begin{tabular}{|c|c|c|c|}\hline
\nsp 1&\nsp 1&\nsp 2& 2\\\hline
\nsp 1&\nsp 1&\nsp 2& 2\\\hline
 3&3& 4&4\\\hline
\nsp 3&\nsp 3&\nsp 4&4\\\hline
\end{tabular}\quad
\begin{tabular}{|c|c|c|c|}\hline
\nsp 5&\nsp 5&\nsp 6& 6\\\hline
\nsp 5&\nsp 5&\nsp 6& 6\\\hline
 7&7& 8&8\\\hline
\nsp 7&\nsp 7&\nsp 8&8\\\hline
\end{tabular}\quad
\begin{tabular}{|c|c|c|c|}\hline
5& 5& 6& 6\\\hline
5&5& 6& 6\\\hline
 7&7& 8&8\\\hline
 7& 7& 8&8\\\hline
\end{tabular}
\end{center}
More formally, we could color with an entry in $\{0,1\}^3$ defined as
\[
\tau(x,y,z)=\left((x \!\!\mod 2a)/a,(y \!\!\mod 2b)/b,(z \!\!\mod 2c)/c\right).
\]
with ``$/$'' indicating integer division with no remainder. If two cuboids rooted at $(x,y,z)$ and $(x',y',z')$ touched via \eqref{touchx}, we would have
\[
x-x'=\pm a
\]
which remains true modulo $2a$. Consequently, one of the reduced coordinates must lie in $\{0,\dots a-1\}$ and the other in $\{a,\dots 2a-1\}$, showing that the first entry of $\tau$ would differ. Identical reasoning from  \eqref{touchy}
or \eqref{touchz} shows the claim.
\end{proof}

\begin{remar}
Ivana Atanasovska, Thomas Barnholdt, Niklas Hjuler and
Mikkel Strunge found the upper bounds for $2\times 2\times 1$, $3 \times 3\times 1$ and $4\times 4\times 1$ as students in the course ``Experimental mathematics''  at University of Copenhagen in 2014.
 One notes how the 7-colorings given above work because there is room for four shaded regions of width 5 and 7, respectively, because $4\cdot 5<7\cdot 3$ and $4\cdot 7= 7\cdot 4$. Atanasovska et al essentially noted in their course work that there would not be room for four shaded regions in a similar attempt for $5\times 5\times 1$, because $4\cdot 9>7\cdot 5$.
 
We realized that there is also the $6$-coloring given in  (e), and that such colorings generalize as given, aided by computer experiments as outlined in Section \ref{periodic}. 
\end{remar}

We summarize below the cases now solved.

\begin{corol}
\begin{eqnarray*}
\chi_1([1,1,1])&=&2\\
\chi_1([2,1,1])&=&3\\
\chi_1([n,1,1])&=&4, \qquad n\geq 3\\
\chi_1([2,2,1])&=&5\\
\chi_1([n,2,1])&=&6, \qquad n\geq 5.
\end{eqnarray*}
\end{corol}

Our analysis shows that the list of situations leading to a chromatic number $\leq 4$ is exhaustive. For $5$, the additional cases
\begin{equation}\label{fiveopen}
3\times 2\times1,3\times 3\times1,4\times 2\times1
\end{equation}
are the only possibilities.

\begin{corol}
\begin{eqnarray*}
\max_{a}\chi_1([a,1,1])&=&4\\
\max_{a,b}\chi_1([a,b,1])&=&\max_{a}\chi_1([a,a,1])\in\{6,7,8\}\\
\max_{a,b,c}\chi_1([a,b,c])&=&\max_{a}\chi_1([a,a,a])\in\{6,7,8\}
\end{eqnarray*}
and we have
\begin{center}
\begin{tabular}{|c||c|c|c|c|c|}\hline
$a$&$1$&$2$&$3$&$4$&$\geq 5$\\\hline\hline
$\chi_1([a,a,a])$&$2$&$6,7,8$&$6,7,8$&$6,7,8$&$6,7,8$\\\hline
$\chi_1([a,a,1])$&$2$&$5$&$5,6,7$&$6,7$&$6,7,8$\\\hline
$\chi_1([a,1,1])$&$2$&$3$&$4$&$4$&$4$\\\hline
\end{tabular}
\end{center}
\end{corol}

The maxima ranging over several entries is dominated by the maxima over just one by Lemma \ref{increases}.

\section{$\chi_2$}

We now address $\chi_2$ with special emphasis on the case $\chi_2([a,b,1])$. We do not know whether $\chi_2$ grows with the entries as $\chi_1$ and consider it improbable as we expect eccentric cuboids to give more freedom than  cubes, even of large size, but we have no counterexamples at this time.

We can, of course, bound $\chi_2$ below by $\chi_1$ and use Lemma \ref{increases} with Proposition \ref{lowerbounds} to give general lower bounds for $\chi_2$. In particular, we have that $\chi_2([a,b,c])\geq 6$ for $[a,b,c]\geq [5,2,1],[4,3,1]$ or $[2,2,2]$. In most cases, this remains  our best  lower bound. We can do better precisely in the following cases.

\begin{propo}\label{lowerboundsii}
\begin{eqnarray*}
\chi_2([a,1,1])&\geq& 5\text{ for all }a\geq 2\\
\chi_2([a,1,1])&\geq& 6\text{ for all }a\geq 5\\
\chi_2([4,2,1])&\geq& 6.
\end{eqnarray*}
\end{propo}
\begin{proof}
We start by considering the building
\begin{center}
\includegraphics[width=6cm]{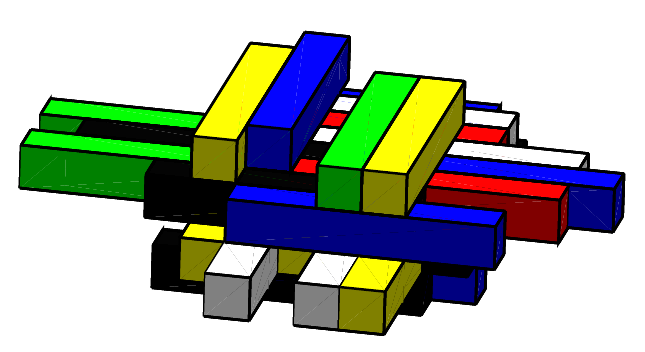}
\end{center}
with 25 $6\times 1\times 1$ cuboids (\ref{611}), requiring six colors. It is easy to see from the exploded view
\begin{center}
\includegraphics[width=6cm]{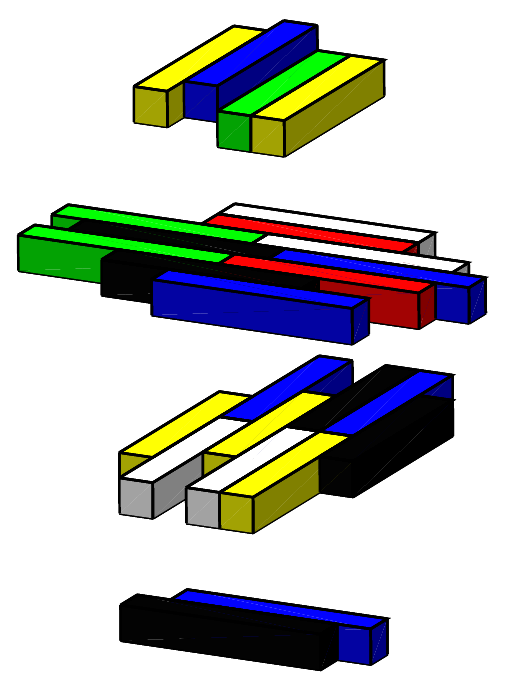}
\end{center}
that a contact graph also requiring six colors can be obtained with cuboids of dimension $a\times 1\times1$ for any $a>6$ by simply extending all cuboids in a free direction. For $a\in\{2,3,4,5\}$ we present configurations requiring five or six colors as indicated. For $5\times 1\times 1$ we can still get to six colors with 98 cuboids (\ref{511}) in
\begin{center}
\includegraphics[width=6cm]{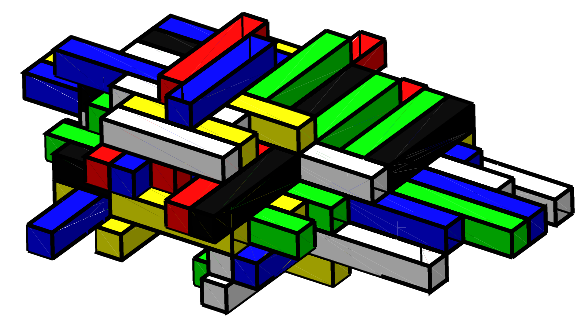}
\end{center}
but this building is not readily extendable, as several cuboids are constrained at both ends by other cuboids:
\begin{center}
\includegraphics[width=6cm]{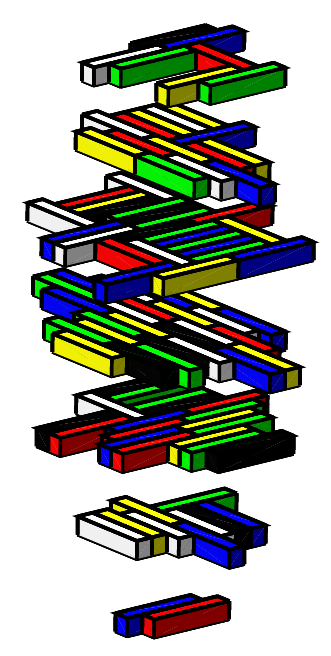}
\end{center}

We can construct configurations requiring five colors for $3\times 1\times 1$ and $4\times 1\times 1$ fairly easily; minimal examples found are
\begin{center}
\includegraphics[width=3cm]{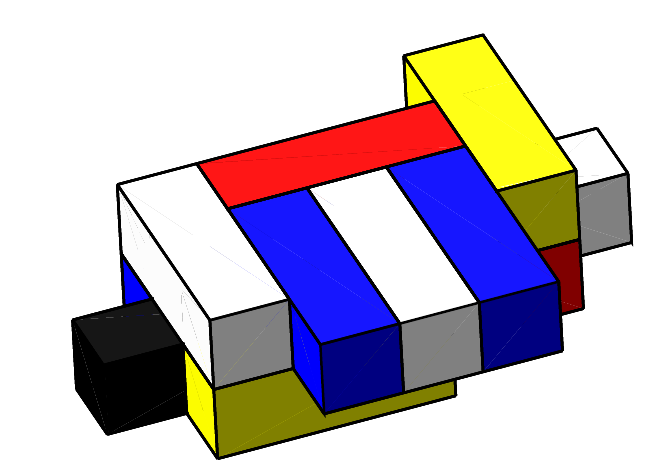}\qquad \includegraphics[width=3.5cm]{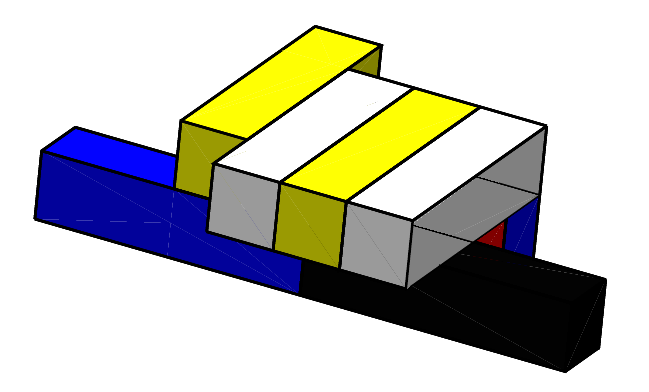}
\end{center}
with 11 and 8 cuboids respectively (\ref{311-2} and \ref{411-2}), or in exploded view
\begin{center}
\includegraphics[width=3cm]{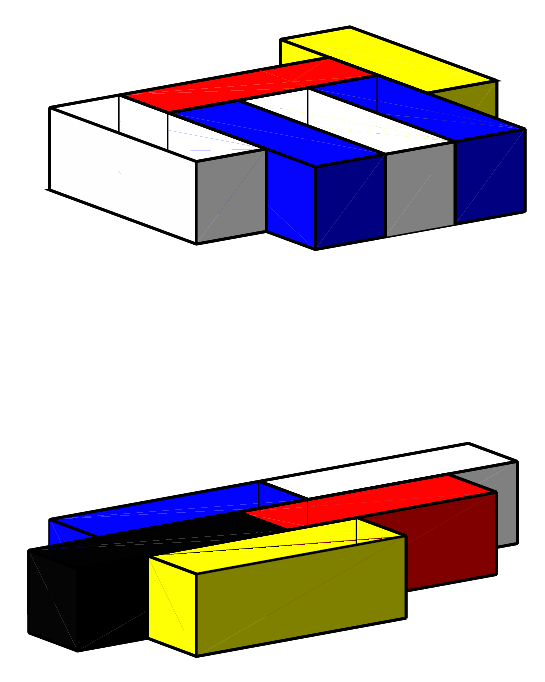}\qquad \includegraphics[width=3.5cm]{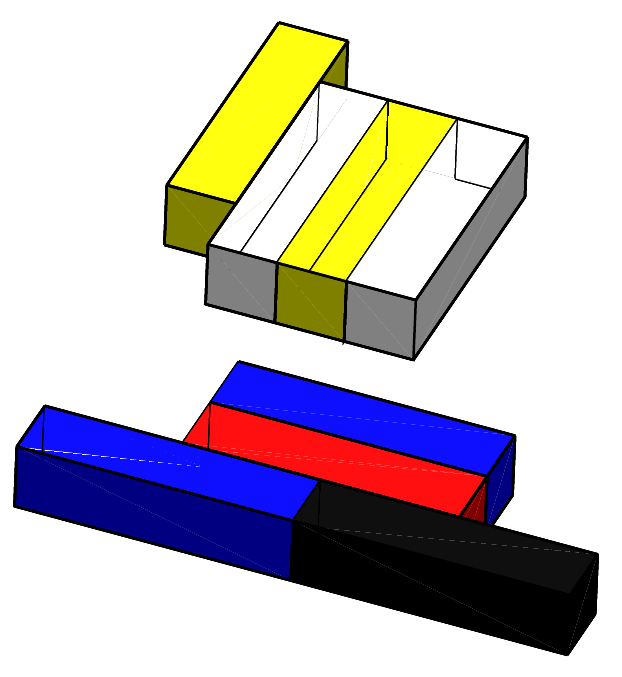}
\end{center}

Getting to five colors with $2\times 1\times 1$ is significantly harder, but examples have been found. The smallest one known to us is
\begin{center}
\includegraphics[width=6cm]{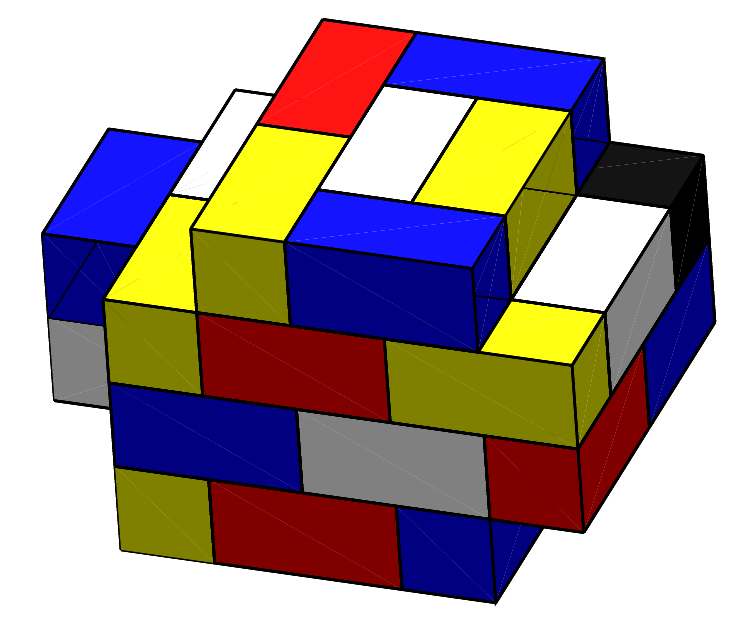}
\end{center}
with 36 cuboids (\ref{211}),  or in exploded view
\begin{center}
\includegraphics[width=5cm]{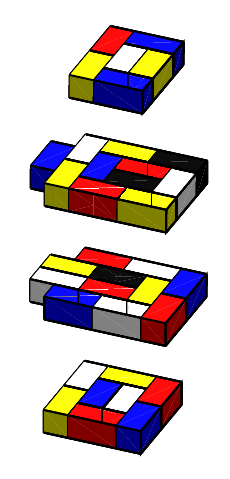}
\end{center}
Finally, we have examples with $4\times 2\times 1$ requiring six colors. The smallest one known to us is
\begin{center}
\includegraphics[width=6cm]{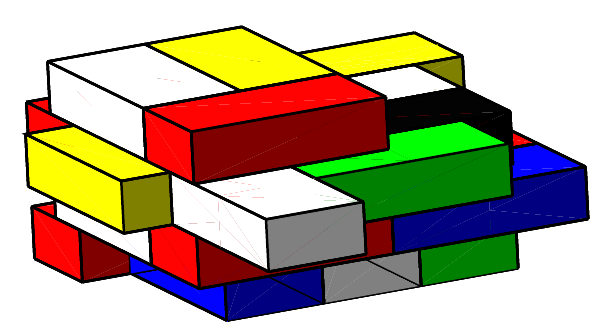}
\end{center}
with 23 cuboids (see \ref{421}),  or in exploded view
\begin{center}
\includegraphics[width=6cm]{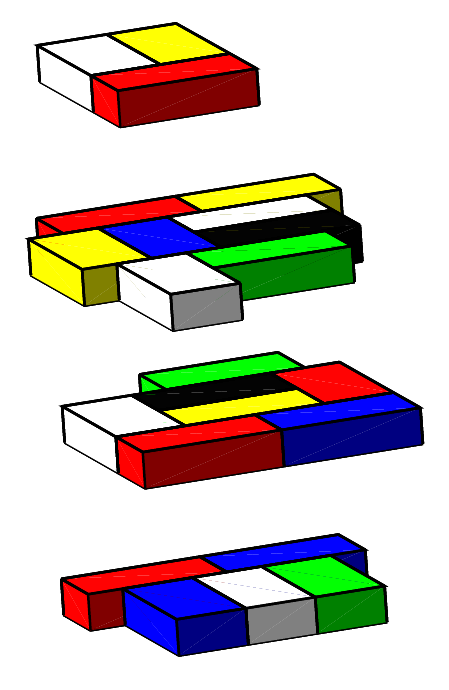}
\end{center}
\end{proof}

\begin{remar}\label{byhand}
We recall again that all of these configurations are critical in the sense that if any one cuboid is removed, the chromatic number goes down by one. This does not rule out the existence of examples having fewer cuboids.

The first examples of buildings with $4\times 2\times1$ (corresponding to the most iconic dimensions of a LEGO brick) requiring 6 colors were found in 2012 by Nicolas Bru Frantzen,
Mikkel B\o{}hlers Nielsen and 
Alex Voigt Hansen as students in the course ``Experimental Mathematics'' at University of Copenhagen. An example which grew out of their work, while not minimal. is worth pointing out as the lack of 5-colorability can rather easily be seen by symmetry. One notes that the configuration 
\begin{center}
\includegraphics[width=6cm]{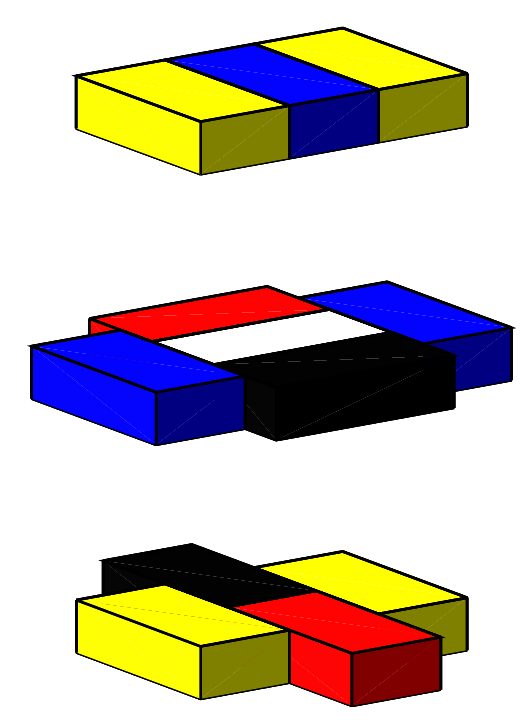}
\end{center}
can be 5-colored, but only when the lowest level has a repeated color as indicated. The 24-cuboid configuration 
\begin{center}
\includegraphics[width=6cm]{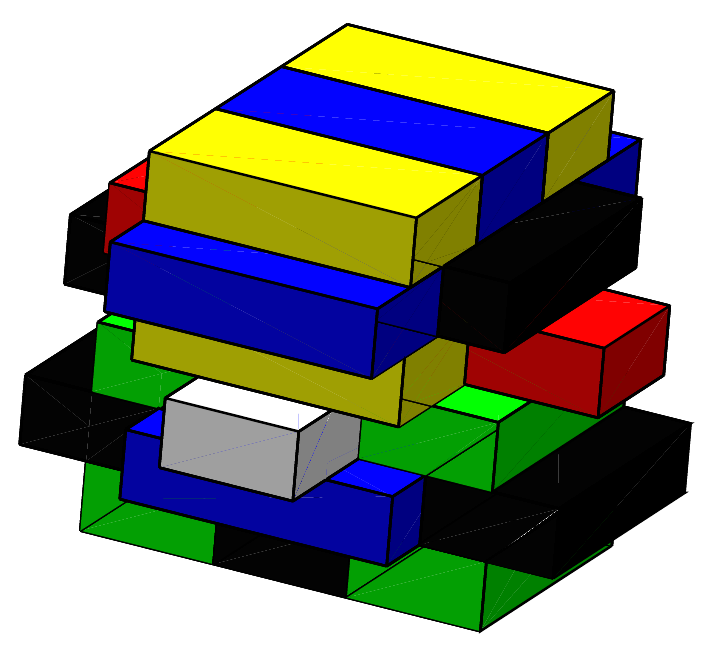}
\end{center}
obtained by adding another copy (mirrored in the $Z$ plane and rotated 90$^\circ$ in the $XY$ plane)  will each have 3 colors in the critical layer, and they must all differ.
\end{remar}

We now present upper bounds.

\begin{theor}\label{boundii}\mbox{}
\begin{enumerate}[(a)]
\item $\chi_2([2,1,1])\leq \perco{2}{2}{1}{1}{10}{10}{2}=5$.
\item For all $a,b$, we have $\chi_2([a,b,1])\leq 8$.
\item For all $a,b,c$, with both $a$ and $b$ odd, we have $\chi_2([a,b,c])\leq 8$.
\item For all $a,b,c$, we have $\chi_2([a,b,c])\leq 16$.
\end{enumerate}
\end{theor}
\begin{proof}
To show that there is a periodic coloring for $2\times1\times1$, we note that any cuboid  $C \in \myC_2([2,1,1])$ intersects exactly one $1\times 1\times1$ cube rooted at $(x,y,z)$ with $x+y+z$ even. Thus, we can specify a coloring by assigning  colors to those cubes. Formally, the coloring is then defined as
\[
\kappa(x,y,z,\pi)=\begin{cases}\kappa_0(x,y,z)&x+y+z\equiv 0\mod 2\\\kappa_0(x+1,y,z)&x+y+z \equiv 1\mod 2,\ \pi=\id\\\kappa_0(x,y+1,z)&x+y+z \equiv 1\mod 2,\ \pi=\tau
\end{cases}
\]
with $\kappa_0$ defined only on triples with even sum.

The coloring is given by employing two levels
\begin{center}
\begin{tabular}{|c|c|c|c|c|c|c|c|c|c|}\hline
\nsp 1&\nsp{} & 2&\nsp{}&3&&4&\nsp{}&5&\nsp{}\\\hline
\nsp{} &3&&4&\nsp{}&5&\nsp{}&\nsp 1&\nsp{}&2\\\hline
4&\nsp{}&5&\nsp{}&\nsp 1&\nsp{}&2&\nsp{}&3&\\\hline
\nsp{}&\nsp 1&\nsp{}&2&\nsp{}&3&&4&\nsp{}&5\\\hline
2&\nsp{}&3&&4&\nsp{}&5&\nsp{}&\nsp 1&\nsp{}\\\hline
&4&\nsp{}&5&\nsp{}&\nsp 1&\nsp{}&2&\nsp{}&3\\\hline
5&\nsp{}&\nsp 1&\nsp{}&2&\nsp{}&3&&4&\nsp{}\\\hline
\nsp{}&2&\nsp{}&3&&4&\nsp{}&5&\nsp{}&\nsp 1\\\hline
 3&&4&\nsp{}&5&\nsp{}&\nsp1&\nsp{}&2&\nsp{}\\\hline
\nsp{}&5&\nsp{}&\nsp1&\nsp{}&2&\nsp{}&3&&4\\\hline
\end{tabular}
\end{center}
and
\begin{center}
\begin{tabular}{|c|c|c|c|c|c|c|c|c|c|}\hline
&4&\nsp{}&5&\nsp{}&\nsp 1&\nsp{} & 2&\nsp{}&3\\\hline
5&\nsp{}&\nsp 1&\nsp{}&2&\nsp{} &3&&4&\nsp{}\\\hline
\nsp{}&2&\nsp{}&3&&4&\nsp{}&5&\nsp{}&\nsp 1\\\hline
3&&4&\nsp{}&5&\nsp{}&\nsp 1&\nsp{}&2&\nsp{}\\\hline
\nsp{}&5&\nsp{}&\nsp 1&\nsp{}&2&\nsp{}&3&&4\\\hline
\nsp 1&\nsp{}&2&\nsp{}&3&&4&\nsp{}&5&\nsp{}\\\hline
\nsp{}&3&&4&\nsp{}&5&\nsp{}&\nsp 1&\nsp{}&2\\\hline
4&\nsp{}&5&\nsp{}&\nsp 1&\nsp{}&2&\nsp{}&3&\\\hline
\nsp{}&\nsp1&\nsp{}&2&\nsp{}& 3&&4&\nsp{}&5\\\hline
2&\nsp{}&3&&4&\nsp{}&5&\nsp{}&\nsp1&\nsp{}\\\hline
\end{tabular}
\end{center}
We get (b) by applying the four color theorem in even and odd vertical layers separately as already noted in Section \ref{partitions}, and (d) follows from combining Theorem \ref{boundi}(h) with Lemma \ref{partitioning}(d)
For (c), we argue along the lines of Theorem \ref{boundi}(h), and  color with an entry in $\{0,1\}^3\simeq \{1,\dots,8\}$ defined as
\[
\kappa(x,y,z)=\left(x \!\!\mod 2,y \!\!\mod 2,(z \!\!\mod 2c)/c\right).
\]
If two cuboids rooted at $(x,y,z)$ and $(x',y',z')$ touched via \eqref{touchx}, we would have
\[
x-x'\in\{\pm a ,\pm b\}
\]
showing that the colors differ. The same argument applies if    \eqref{touchy} holds, and in the case 
or \eqref{touchz} we argue just as in  Theorem \ref{boundi}(h).
\end{proof}

\begin{remar}
Ivana Atanasovska, Thomas Barnholdt, Niklas Hjuler and
Mikkel Strunge found the upper bound for $2\times 1\times 1$ as students in the course ``Experimental mathematics''  in 2014.
\end{remar}

\begin{corol}
\[
\chi_2([2,1,1])=5.
\]
\end{corol}

\begin{corol}
\begin{eqnarray*}
\max_{a}\chi_2([a,1,1])&\in&\{6,7,8\}\\
\max_{a,b}\chi_2([a,b,1])&\in&\{6,7,8\}\\
\max_{a,b,c}\chi_2([a,b,c])&\in&\{6,7,\dots,16\}
\end{eqnarray*}
and for small cuboids, we have
\begin{center}
\begin{tabular}{|c||c|c|c|c|c|}\hline
$a$&$1$&$2$&$3$&$4$&$\geq 5$\\\hline\hline
$\chi_2([a,1,1])$&$2$&$5$&$5,6,7,8$&$5,6,7,8$&$6,7,8$\\\hline
$\chi_2([a,2,1])$&&$5$&$5,6,7,8$&$6,7,8$&$6,7,8$\\\hline
$\chi_2([a,3,1])$&&&$5,6,7$&$6,7,8$&$6,7,8$\\\hline
$\chi_2([a,4,1])$&&&&$6,7$&$6,7,8$\\\hline
$\chi_2([a,5,1])$&&&&&$6,7,8$\\\hline
\end{tabular}
\end{center}
\end{corol}

We end this section by a brief discussion of a few aspects of the case $\chi_2([a,b,c])$ with $c>1$. Note first that in this case, the lack of symmetry prevents us from always arranging that $a\geq b\geq c$, but we will always require $a\geq b$. Just as in Lemma \ref{increases}, we can show

\begin{lemma}\label{increacesaltchiii}
When $c\leq c'$ then $\CG_2([a,b,c])\subseteq \CG_2([a,b',c'])$.
\end{lemma}

\begin{propo}\label{nstdlowerboundsii}
When $c>1$, we have 
\[
\chi_2([a,b,c])\geq 6
\]
unless $a=b=1$.
\end{propo}
\begin{proof}
We assume that $a\geq b$; by Lemma \ref{increacesaltchiii} we may assume that $c=2$. If then $b\geq 2$ we get
\[
\chi_2([a,b,c])\geq \chi_1([a,b,c])\geq \chi_1([2,2,2])\geq 6
\]
using Lemma \ref{increases} and Proposition \ref{lowerbounds}. If $b=1$ and $a\geq 5$ we argue similarly that 
\[
\chi_2([a,b,c])\geq \chi_1([a,b,c])\geq \chi_1([5,1,2])\geq 6
\]
so the remaining cases are $b=1$ and $a\in\{2,3,4\}$. Here 
we have found the examples needing 6 colors
\begin{center}
\includegraphics[width=7cm]{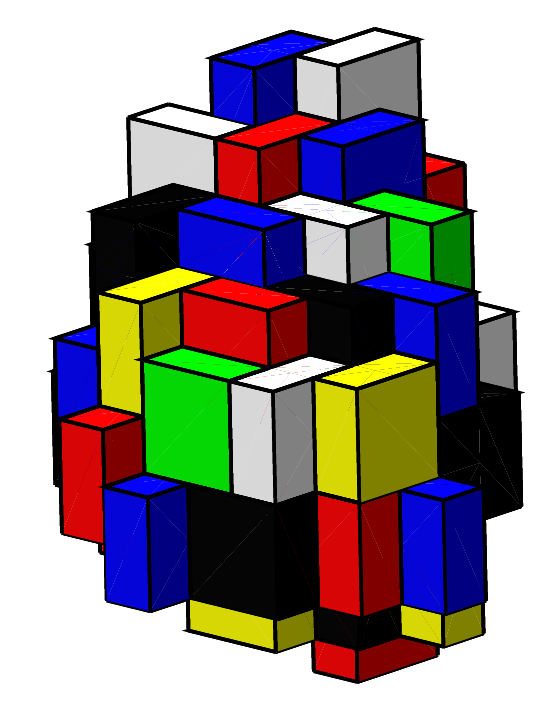}
\end{center}
with 61 $2\times 1\times 2$ cuboids (see \ref{212}), and 
\begin{center}
\includegraphics[width=4cm]{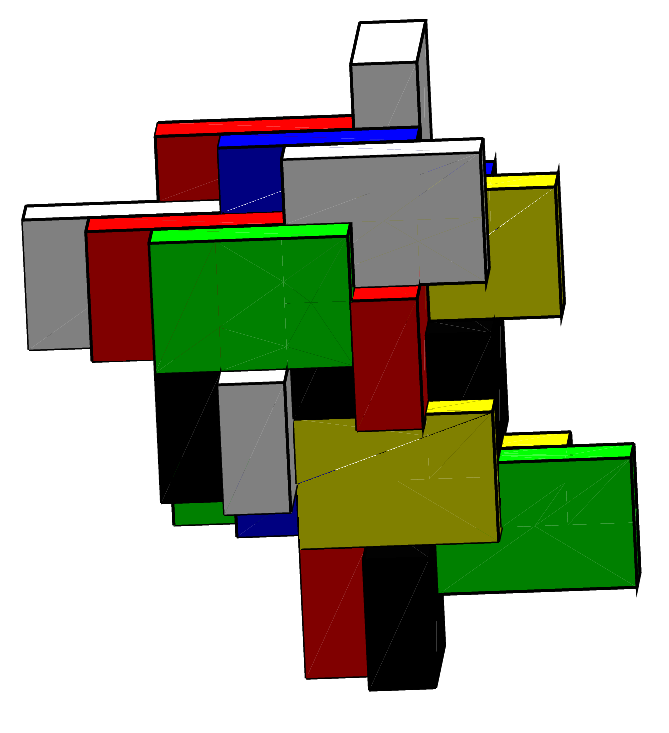}\qquad \includegraphics[width=4cm]{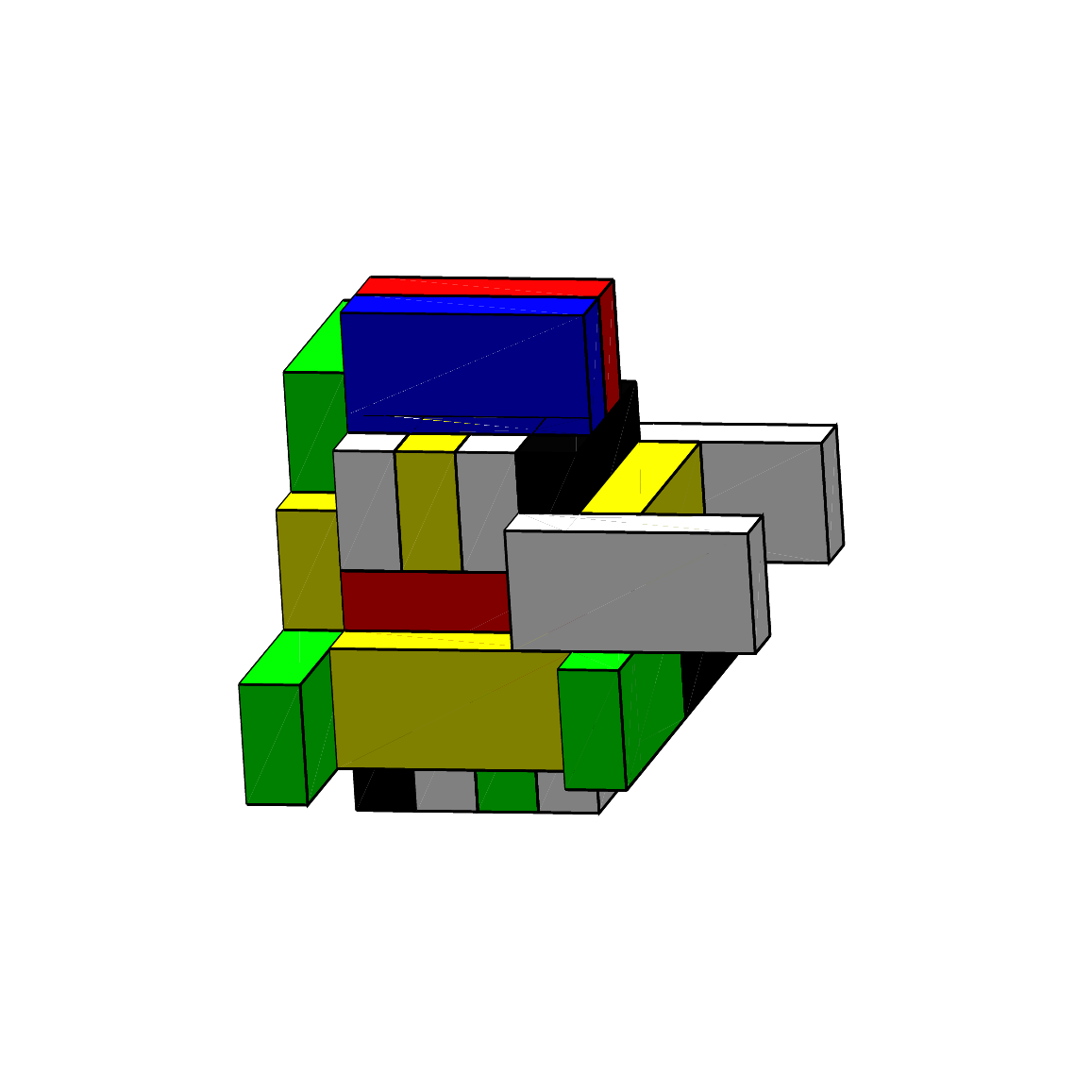}
\end{center}
with 25 and 26, respectively, cuboids of size $3\times 1\times 2$ and $4\times 1\times 2$ (see \ref{312} and \ref{412}).
\end{proof}

 We can improve the general upper bound $\chi_2([a,b,c])\leq 16$ for $c>1$ in many cases.

  \begin{lemma}
$ \nv_2([a,b,c])=\nv_2([a,b,2])$ for all $c>1$.
 \end{lemma}
 \begin{proof}
 Since all elements in $\myC_2([a,b,c])$ have the same height, any neighbor of a fixed cuboid is either vertically aligned with it, or protrudes upwards or downwards. Without creating collisions, we can hence adjust the vertical position of each neighbor so that  any protruding neighbor overlaps vertically with length $c/2$.
 \end{proof}

As explained in Section \ref{maxmin}, improved estimates can be obtained by computing all positions of neighboring cuboids and creating a collision  graph for them. We then compute $\nv_2([a,b,2])$ as the independence number of this graph:
  
\begin{center}
\begin{tabular}{c||c|c|c|c|c|c|c|c|c|c|c|c|c|c|c|}
&1&2&3&4&5&6&7&8&9&10&11&12&13&14&15\\\hline\hline
1&5 & \nsp{8} &{11}  & \nsp{14} & 17 & 20 & 23 & 26 & 29 & 32 & 35 & 38 & 41 & 44 & 47 
\\\hline
2& & 7 & \nsp{9} & \nsp{10}  & \nsp{12}   & \nsp{13}    & 15 & 16 & 18 & 19 & 20 & 22 & 23 & 25 & 26 
\\\hline 
3& & & 7 & \nsp{9} &{10}  & \nsp{10}  &{12}   & \nsp{13}   &{13} & 15 & 16 & 16 & 18 & 18 & 19 
 \\\hline 
4& & & & 7 & \nsp{9} & \nsp{10}  & \nsp{10}  & \nsp{10}  & \nsp{12}   &  \nsp{13} & \nsp{13} &  \nsp{13} & 15 & 16 & 16 
 \\\hline 
5& & & & & 7 & \nsp{9} &{10}  & \nsp{10}  &{10}  & \nsp{10}  & {12}   &  \nsp{13} &  {13} &  \nsp{13} &  {13} 
 \\\hline 
6& & & & & & 7 & \nsp{9} & \nsp{10}  & \nsp{10}  & \nsp{10}  & \nsp{10}  & \nsp{10}  & \nsp{12}   &  \nsp{13} &  \nsp{13} 
 \\\hline 
7& & & & & & & 7 & \nsp{9} &{10}  & \nsp{10}  & {10}  & \nsp{10}  & {10}  & \nsp{10}  & {12}   
 \\\hline 
 8&& & & & & & & 7 & \nsp{9} & \nsp{10}  & \nsp{10}  & \nsp{10}  & \nsp{10}  & \nsp{10}  & \nsp{10}  
 \\\hline 
9& & & & & & & & & 7 & \nsp{9} & {10}  & \nsp{10}  & {10}  & \nsp{10}  & {10}  
 \\\hline 
10& & & & & & & & & & 7 & \nsp{9} & \nsp{10}  & \nsp{10}  & \nsp{10}  & \nsp{10}  
 \\\hline 
 11&& & & & & & & & & & 7 & \nsp{9} & {10}  & \nsp{10}  & {10}  
 \\\hline 
12& & & & & & & & & & & & 7 & \nsp{9} & \nsp{10}  & \nsp{10}  
 \\\hline 
13& & & & & & & & & & & & & 7 & \nsp{9} & {10}  
 \\\hline 
14& & & & & & & & & & & & & & 7 & \nsp{9} 
 \\\hline 
15& & & & & & & & & & & & & & & 7 
 \\\hline 
\end{tabular}
\end{center}
The  entries in shaded cells are the ones improving upper bounds already presented.

\section{$\chi_3$}

We recall that we have shown that
\[
\chi_1([a,1,1])= 4<5\leq \chi_2([a,1,1])
\]
for all $a$, and when we exploit the asymmetry in our definition of $\chi_2$, we also get from Proposition \ref{nstdlowerboundsii} that 
\[
\chi_1([2,1,2])=5<6\leq \chi_2([2,1,2])
\]
while
\[
\chi_2([2,2,1])=5<6\leq \chi_2([2,1,2])\leq \chi_3([2,1,2])=\chi_3([2,2,1]).
\]
But in no instance $a > b > c$  can we rule out that $\chi_1([a,b,c])=\chi_2([a,b,c])=\chi_3([a,b,c])$, even though it would be very counterintuitive to us if this were the case.

These phenomena explain  that except for a few  dimensions, at least when one dimension is even, our best knowledge of  $\chi_3([a,b,c])$ is in fact derived from knowledge of  $\chi_2([a,b,c])$ or $\chi_1([a,b,c])$. Lower bounds are of course transferred directly.

\begin{theor}
We have
\[
\chi_3([a,b,c])\geq 6
\]
for all $(a,b,c)\not \in\{(1,1,1),(2,1,1),(3,1,1)\}$.
\end{theor}
\begin{proof}
By Proposition \ref{lowerbounds} we have $\chi_3([a,b,c])\geq\chi_1([a,b,c])\geq  6$ except when $b=c=1$ or
\begin{equation}\label{opens}
(a,b,c)\in\{(2,2,1),(3,2,1),(3,3,1),(4,2,1)\}.
\end{equation}
In the case $b=c=1$ we get
\begin{eqnarray*}
\chi_3([a,1,1])&\geq& \chi_2([a,1,1])\ \geq\   6
\end{eqnarray*}
whenever $a\geq 5$ by Proposition \ref{lowerboundsii}, and for the cases in \eqref{opens} we can permute the entries to get
\[
\chi_3([a,b,1])=\chi_3([a,1,b])\geq \chi_2([a,1,b])\geq 6
\]
by Proposition \ref{nstdlowerboundsii}.

Finally, we have the example
\begin{center}
\includegraphics[width=6cm]{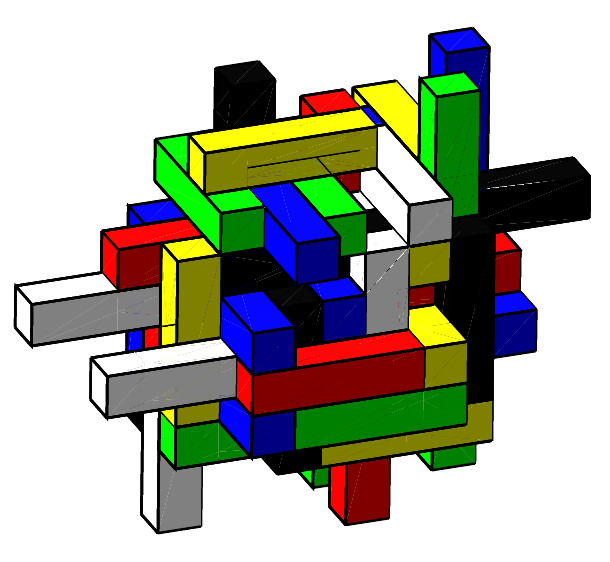}
\end{center}
with 56 cuboids (\ref{411-3}), showing that $\chi_3([4,1,1])\geq 6$.
\end{proof}

Our upper bounds at this level of freedom appear rather loose in most cases.

\begin{propo}\mbox{}
\begin{enumerate}[(a)]
\item $\chi_3([2,1,1])\leq  \perco{3}{2}{1}{1}{12}{12}{12}=6$.
\item $\chi_3([a,b,c])\leq 8$ when all $a,b,c$ are odd.
\item $\chi_3([a,1,1])\leq 12$ for all $a$.
\item $\chi_3([a,b,1])\leq 24$ for all $a,b$.
\item $\chi_3([a,b,c])\leq 24$ for all $a,b,c$ when $a=b$ or $b=c$.
\item $\chi_3([a,b,c])\leq 48$ for all $a,b,c$.
\end{enumerate}
\end{propo}
\begin{proof}
As for $\chi_2([2,1,1])$ we note that any $2\times 1\times 1$ cuboid overlaps with exactly one $1\times1\times1$ cuboid with coordinates summing to an even number. We set out to color these periodically in a $12\times 12\times 12$ grid, where we must deal with the fact that a cuboid placed due to a number situated in level $\ell$ may now extend into level $\ell-1$ and $\ell+1$. In the $\chi_2$ case, it is allowed that two $1\times1\times 1$ cuboids which may obtain the same color in adjacent layers share a side, but when two cuboids are adjacent in the same layer, they may only share a point. The formal description of the coloring is
\[
\kappa(x,y,z,\pi)=\begin{cases}\kappa_0(x,y,z)&x+y+z\equiv 0\mod 2\\
\kappa_0(x+1,y,z)&x+y+z\equiv 1\mod 2,\ \pi(1)=1\\
\kappa_0(x,y+1,z)&x+y+z\equiv 1\mod 2,\ \pi(1)=2\\
\kappa_0(x,y,z+1)&x+y+z\equiv 1\mod 2,\ \pi(1)=3
\end{cases}
\]

 We start by
\begin{center}
\begin{tabular}{|c|c|c|c|c|c|c|c|c|c|c|c|}\hline
\nsp{}&3&\nspi{}&6&\nsp{}&3&\nspi{}&6&\nsp{}&3&\nspi{}&6\\\hline
2&\nspii{}&5&\nsp{}&2&\nspii{}&5&\nsp{}&2&\nspii{}&5&\nsp{}\\\hline
\nsp{}&4&\nsp{}&\nsp{1}&\nsp{}&4&\nsp{}&\nsp{1}&\nsp{}&4&\nsp{}&\nsp{1}\\\hline
3&\nspi{}&6&\nsp{}&3&\nspi{}&6&\nsp{}&3&\nspi{}&6&\nsp{}\\\hline
\nspii{}&5&\nsp{}&2&\nspii{}&5&\nsp{}&2&\nspii{}&5&\nsp{}&2\\\hline
4&\nsp{}&\nsp{1}&\nsp{}&4&\nsp{}&\nsp{1}&\nsp{}&4&\nsp{}&\nsp{1}&\nsp{}\\\hline
\nspi{}&6&\nsp{}&3&\nspi{}&6&\nsp{}&3&\nspi{}&6&\nsp{}&3\\\hline
5&\nsp{}&2&\nspii{}&5&\nsp{}&2&\nspii{}&5&\nsp{}&2&\nspii{}\\\hline
\nsp{}&\nsp{1}&\nsp{}&4&\nsp{}&\nsp{1}&\nsp{}&4&\nsp{}&\nsp{1}&\nsp{}&4\\\hline
6&\nsp{}&3&\nspi{}&6&\nsp{}&3&\nspi{}&6&\nsp{}&3&\nspi{}\\\hline
\nsp{}&2&\nspii{}&5&\nsp{}&2&\nspii{}&5&\nsp{}&2&\nspii{}&5\\\hline
\nsp{1}&\nsp{}&4&\nsp{}&\nsp{1}&\nsp{}&4&\nsp{}&\nsp{1}&\nsp{}&4&\nsp{}\\\hline
\end{tabular}
\end{center}
where the gray cells are those that can be occupied by a cuboid placed horizontally colored with color 1, and the other shaded cells are those that can be occupied by a cuboid  placed vertically and colored 1 by a specification in one of the neighboring layers. The next level is obtained by shifting one level down and two across (like a knight in chess) to obtain
\begin{center}
\begin{tabular}{|c|c|c|c|c|c|c|c|c|c|c|c|}\hline
5&\nspii{}&2&\nspi{}&5&\nspii{}&2&\nspi{}&5&\nspii{}&2&\nspi{}\\\hline
\nspii{}&\nspii{1}&\nspii{}&4&\nspii{}&\nspii{1}&\nspii{}&4&\nspii{}&\nspii{1}&\nspii{}&4\\\hline
6&\nspii{}&3&\nsp{}&6&\nspii{}&3&\nsp{}&6&\nspii{}&3&\nsp{}\\\hline
\nspii{}&2&\nspi{}&5&\nspii{}&2&\nspi{}&5&\nspii{}&2&\nspi{}&5\\\hline
\nspii{1}&\nspii{}&4&\nspii{}&\nspii{1}&\nspii{}&4&\nspii{}&\nspii{1}&\nspii{}&4&\nspii{}\\\hline
\nspii{}&3&\nsp{}&6&\nspii{}&3&\nsp{}&6&\nspii{}&3&\nsp{}&6\\\hline
2&\nspi{}&5&\nspii{}&2&\nspi{}&5&\nspii{}&2&\nspi{}&5&\nspii{}\\\hline
\nspii{}&4&\nspii{}&\nspii{1}&\nspii{}&4&\nspii{}&\nspii{1}&\nspii{}&4&\nspii{}&\nspii{1}\\\hline
3&\nsp{}&6&\nspii{}&3&\nsp{}&6&\nspii{}&3&\nsp{}&6&\nspii{}\\\hline
\nspi{}&5&\nspii{}&2&\nspi{}&5&\nspii{}&2&\nspi{}&5&\nspii{}&2\\\hline
4&\nspii{}&\nspii{1}&\nspii{}&4&\nspii{}&\nspii{1}&\nspii{}&4&\nspii{}&\nspii{1}&\nspii{}\\\hline
\nsp{}&6&\nspii{}&3&\nsp{}&6&\nspii{}&3&\nsp{}&6&\nspii{}&3\\\hline
\end{tabular}
\end{center}
where we now color the vertical positions magenta to match the coloring above. Note that the cuboids colored magenta in the first layer are exactly above ones in the second layer, and the cuboids colored grey in the second layer exactly below ones in the first layer. Shifting 10 more times completes the cycle.

For (b), we argue as in Proposition \ref{boundi}(h), and  color with an entry in $\{0,1\}^3$ defined as
\[
\kappa(x,y,z)=\left(x\!\! \mod 2,y\!\!\mod 2,z\!\!\mod 2\right)
\]
If two cuboids rooted at $(x,y,z)$ and $(x',y',z')$ touched via \eqref{touchx}, we would have
\[
x-x'\in\{\pm a,\pm b,\pm c\},
\]
showing that the first entry of $\tau$ would differ. Identical reasoning from  \eqref{touchy}
or \eqref{touchz} shows the claim.

We get (c) from  Proposition \ref{boundi}(c) with Lemma \ref{partitioning}(b), (d) from  Proposition \ref{boundii}(b) with with Lemma \ref{partitioning}(b), (e) from  Proposition \ref{boundi}(h) with Lemma \ref{partitioning}(b) and (f)  from  Proposition \ref{boundi}(h) with Lemma \ref{partitioning}(b).
\end{proof}

\begin{remar}
The coloring in (a) was discovered by Niklas Hjuler in 2018.
\end{remar}

The proof provided for (b) also shows that 8 colors suffice when all dimensions are odd, even when the cuboids are allowed to vary freely, so we note explicitly for possible later use:

\begin{corol}\label{allodd}
Any contact graph of an integral  cuboid configuration where all cuboids have \textbf{odd} sides can be 8-colored.
\end{corol}

 The authors are grateful to St\'ephane Bessy and Jean-S\'ebastian Sereni (\cite{sbjss:pc}) for showing them that the argument of  \cite{bessygoncalvessereni} can be amended to show that there is a building with all cuboids layered and of size $1\times 1\times (2n+1)$, which requires eight colors.

As in the previous section, we can in some cases use $\nv_3$ to obtain better bounds. With $c=1$ and $a\geq b$, we have

\begin{center}
\begin{tabular}{c||c|c|c|c|c|c|c|c|}
&1&2&3&4&5&6&7&8\\\hline\hline
1&3 & 7 &  11 & 15 & 19 & 23 & 27 & 31 
\\\hline
2&& \nsp 11 & \nsp 16 & \nsp 21 & 26 & 31 & 34 & 39 
\\\hline
3&&& \nsp 19 & \nsp 23 & 29 & 35 & 40 & 45 
\\\hline
4&&&& 27 & 31 & 37 & 42 & 47 
\\\hline
5&&&&& 33 & 39 & 44 & 49 
\\\hline
6&&&&&& 41 & 47 & 52 
\\\hline
7&&&&&&& 47 & 53 
\\\hline
 8& &&&&&&& 55\\ \hline
\end{tabular}\end{center}

With $c=2$ and $a\geq b\geq 2$, we have

\begin{center}
\begin{tabular}{c||c|c|c|c|c|c|c|}
&2&3&4&5&6&7&8\\\hline\hline
2&7 & \nsp 10 & \nsp 13 & \nsp 16 &\nsp  19 &\nsp  21 &  24 
\\\hline
3&& \nsp 12 &\nsp  14 & \nsp 17 &\nsp  20 &\nsp  23 &\nsp  25 
\\\hline
4&&& \nsp 15 & \nsp 18 &\nsp  21 &\nsp  24 &\nsp  27 
\\\hline
5&&&&\nsp  18 &\nsp  21 &\nsp  24 &\nsp  26 
\\\hline
6&&&&&  23 &\nsp  25 &\nsp  28 
\\\hline
7&&&&&&  26 &\nsp  29 
\\\hline
8&&&&&&&30 \\ \hline
\end{tabular}\end{center}
which provides much more new information both because our general bounds are worse, but also because  $\nv_3$ grows more slowly due to  the cuboids being less eccentric.

With $c=3$ and $a\geq b\geq 3$, we have

\begin{center}
\begin{tabular}{c||c|c|c|c|c|c|}
&3&4&5&6&7&8\\\hline\hline
3&7 & \nsp 10 & 13 &\nsp  13 & 16 &\nsp  18 
\\\hline
4&& \nsp 12 & \nsp 14 & \nsp 14 & \nsp 17 & \nsp 20 
\\\hline
5&&&  14 &  \nsp 16 & 18 & \nsp 20 
\\\hline
6&&&& \nsp 15 & \nsp 17 & \nsp 20 
\\\hline
7&&&&& 18 & \nsp 21 
\\\hline
8&&&&&& \nsp 21 
\\\hline
\end{tabular}\end{center}

It is easy to see that $\nv_3([2,3,25])\geq 48$, providing an example of a dimension where we cannot rule out $\chi_3=48$. 

\begin{corol}
\begin{eqnarray*}
\max_{a}\chi_3([a,1,1])&\in&\{6,\dots,12\}\\
\max_{a,b}\chi_3([a,b,1])&\in&\{6,\dots,24\}\\
\max_{a,b,c}\chi_3([a,b,c])&\in&\{6,\dots,48\}
\end{eqnarray*}
and we have
\begin{center}
\begin{tabular}{|c||c|c|c|c|c|}\hline
$a$&$1$&$2$&$3$&$4$&$5$\\\hline\hline
$\chi_3([a,1,1])$&$2$&$5,6$&$5,6,7,8$&$6,\dots,12$&$6,7,8$\\\hline
$\chi_3([a,2,1])$&&$6,\dots,12$&$6,\dots,17$&$6,\dots,22$&$6,\dots,24$\\\hline
$\chi_3([a,3,1])$&&&$6,7,8$&$6,\dots,24$&$6,7,8$\\\hline
\end{tabular}
\end{center}
\end{corol}

We certainly do not expect sequences such as $\chi_3([a,3,1])$ to alternate between large and small numbers as $a$ grows, but our results reflect that we have much better estimates when Corollary \ref{allodd} is applicable than when it is not. We see, however, no way to infer bounds on the general case from this.

\section{Closing remarks}

The most striking fact about the previous section is probably  that we did not present any examples where $\chi_3([a,b,c])>6$. This is not for want of trying; indeed, we have searched extensively for such examples, but none have been found. More generally, we note that of all the numbers $\chi_\bullet([a,b,c])$ we have studied here, we only know for a fact that the numbers
\begin{gather}\label{lessthansix}
 \chi_1([a,1,1]),\chi_1([2,2,1]),\chi_2([2,1,1])
\end{gather}
are not equal to 6, since we know they are 5 or less. For the additional numbers
\begin{gather}
 \chi_1([3,2,1]), \chi_1([3,3,1]), \chi_1([4,2,1])\notag\\
 \chi_2([3,1,1]),\chi_2([3,2,1]),\chi_2([4,1,1])\label{canbelessthansix}\\
 \chi_3([2,1,1]),\chi_3([3,1,1])\notag
\end{gather}
6 and 5 are both options, and in all cases not mentioned in \eqref{lessthansix} or \eqref{canbelessthansix} we know that the chromatic number is at least 6, but could be significantly larger.
We single out four instances for which deciding if the chromatic number is or isn't 6 seems particularly important.

\begin{quest}
Is    $\chi_3([2,1,1])=6$? Or is $\chi_3([2,1,1])=5$?
\end{quest}

We can prove that there is no periodic 5-coloring in this case, and  have searched extensively for configurations needing 6 colors without results. 

\begin{quest}
Is   $\chi_1([4,2,1])=6$? Or is $\chi_1([4,2,1])=5$?
\end{quest}

This open question appears to be the most likely to reach a chromatic number of 6.

\begin{quest}
Is  $\chi_1([2,2,2])=6$? Or is $\chi_1([2,2,2])\in\{7,8\}$?
\end{quest}

It would be very interesting even to rule out the option $\chi_1=8$. 

\begin{quest}
Is $\chi_3([5,2,1])=6$? Or is $\chi_3([5,2,1])\in\{7,\dots,24\}$?
\end{quest}

Again, improved upper bounds would be useful.
The cuboid dimensions suggested here appear to us to be promising, since we believe they are large enough to support what we expect is general behavior, whilst being small enough to allow computer-based experimentation. 
It is possible that such an experiment requires computational power beyond what we have on our standard laptops.

Needless to say, it is also interesting to study these types of questions at other dimensions $d\not=3$. It is straightforward to show, using the four color theorem and methods similar to the ones used above, that $\chi_1([a,b])=4$ except that $\chi_1([2,1])=3$ and $\chi_1([1,1])=2$, and that $\chi_2([a,b])=4$ except that $\chi_2([1,1])=2$. We intend to study higher dimensions elsewhere, but note that our methods  at least show that 
\[
d+1\leq \max_{x_1,\dots,x_d}\chi_1([x_1,\dots,x_d])\leq \max_{x_1,\dots,x_d}\chi_d([x_1,\dots,x_d])\leq d!2^d
\]
for all $d$.

%
%

\appendix

\section{Coordinate data}\label{coordinates}

The configurations can be displayed (and then rotated) in Maple by calling \verb!draw! on the lists after running
\begin{verbatim}
with(plots):
with(plottools):
colchoice:=[yellow,red,blue,white,black,green];
draw:=L->display(map(x->display(cuboid([x[1],x[3],x[5]],
        [x[2],x[4],x[6]]),color=colchoice[x[7]]),L),
        scaling=constrained);
\end{verbatim}
To obtain exploded versions of layered configurations, one can use
\begin{verbatim}
drawexplodeZ:=L->display(map(x->display(cuboid([x[1],x[3],5*x[5]],
                          [x[2],x[4],4*x[5]+x[6]]),
                          color=colchoice[x[7]]),L),
                          scaling=constrained);
\end{verbatim}
\subsection{\tt821/2/6}\label{821}
{\tiny \begin{verbatim}
[[-1,1,-8,0,0,1,1],[-1,1,0,8,0,1,3],[1,3,-4,4,0,1,4],[3,5,-4,4,0,1,3],[-3,-1,-4,4,0,1,4],[-5,-3,-4,4,0,1,3],
[-8,0,-1,1,1,2,2],[0,8,-1,1,1,2,5],[-4,4,1,3,1,2,6],[-4,4,3,5,1,2,5],[-4,4,-3,-1,1,2,6],[-4,4,-5,-3,1,2,5]]
\end{verbatim}}

\subsection{\tt311/1/4}\label{311-1}
{\tiny\begin{verbatim}
[[9,10,12,15,11,12,1],[13,14,13,16,11,12,4],[10,13,13,14,11,12,2],[10,11,14,17,11,12,3],[9,12,14,15,10,11,2],
[11,14,13,14,10,11,3],[8,11,13,14,10,11,4],[11,14,15,16,10,11,1],[12,15,14,15,10,11,5],[11,12,14,17,11,12,4],
[12,13,14,17,11,12,3]]
\end{verbatim}}

\subsection{\tt 221/1/5}\label{221}
{\tiny \begin{verbatim}
[[3,5,3,5,6,7,2],[6,8,4,6,5,6,2],[4,6,3,5,7,8,1],[5,7,4,6,6,7,4],[4,6,4,6,5,6,1],[6,8,3,5,7,8,2],[5,7,2,4,6,7,3],
[7,9,3,5,6,7,5],[4,6,2,4,5,6,4],[6,8,2,4,5,6,1],[5,7,3,5,4,5,5]]
\end{verbatim}}

\subsection{\tt 521/1/6}\label{521}
{\tiny \begin{verbatim}[[16,17,17,19,28,33,5],[15,16,19,21,25,30,5],[16,17,19,21,26,31,1],[15,16,18,20,30,35,3],[16,17,19,21,31,36,2],
[15,16,20,22,30,35,6],[15,16,17,19,25,30,1],[16,17,21,23,29,34,5],[15,16,19,21,35,40,4],[14,15,18,20,29,34,2],
[14,15,19,21,34,39,1],[14,15,20,22,29,34,4],[15,16,21,23,25,30,3],[17,18,18,20,27,32,4],[17,18,20,22,27,32,6],[16,17,17,19,23,28,2],
[16,17,21,23,24,29,4],[16,17,19,21,21,26,3],[17,18,18,20,22,27,6],[17,18,20,22,22,27,2],[17,18,16,18,25,30,3],[18,19,19,21,23,28,3],
[17,18,22,24,26,31,3],[15,16,20,22,20,25,2],[15,16,18,20,20,25,4],[14,15,19,21,24,29,3],[18,19,17,19,24,29,2],[16,17,21,23,19,24,5],
[16,17,19,21,16,21,1],[16,17,17,19,18,23,5],[17,18,16,18,20,25,1],[17,18,18,20,17,22,2],[17,18,20,22,17,22,6],[18,19,17,19,19,24,4],
[18,19,19,21,18,23,1],[17,18,22,24,21,26,1],[18,19,21,23,24,29,4],[18,19,21,23,19,24,5],[18,19,19,21,28,33,5],[16,17,23,25,25,30,6],
[16,17,15,17,24,29,4],[15,16,16,18,20,25,3],[16,17,15,17,19,24,6],[17,18,16,18,15,20,3],[14,15,17,19,24,29,6],[18,19,17,19,14,19,6],
[18,19,15,17,16,21,2],[18,19,15,17,21,26,6],[17,18,14,16,17,22,5],[17,18,14,16,22,27,2],[18,19,19,21,13,18,5],[19,20,20,22,22,27,6],
[18,19,21,23,14,19,4],[19,20,18,20,25,30,4],[18,19,17,19,29,34,1],[18,19,15,17,26,31,4],[18,19,21,23,29,34,1],[17,18,22,24,16,21,3],
[19,20,16,18,27,32,6],[17,18,20,22,12,17,2],[17,18,18,20,12,17,4],[17,18,18,20,32,37,6],[17,18,16,18,30,35,2],[17,18,20,22,32,37,3],
[18,19,19,21,33,38,2],[18,19,17,19,34,39,4],[16,17,23,25,20,25,2],[17,18,22,24,31,36,2],[17,18,14,16,27,32,6],[18,19,15,17,31,36,3],
[19,20,16,18,32,37,2],[16,17,21,23,14,19,4],[16,17,17,19,33,38,4],[16,17,15,17,29,34,1],[15,16,16,18,30,35,2],[16,17,21,23,34,39,4],
[16,17,19,21,36,41,1],[16,17,17,19,13,18,6],[15,16,15,17,25,30,5],[17,18,16,18,35,40,1],[17,18,18,20,37,42,3],[17,18,14,16,32,37,4],
[16,17,15,17,34,39,3],[16,17,15,17,14,19,2],[18,19,21,23,34,39,4],[17,18,20,22,37,42,6],[17,18,16,18,10,15,1],[16,17,23,25,30,35,1],
[17,18,14,16,12,17,6],[18,19,15,17,11,16,4],[18,19,17,19,9,14,2],[17,18,22,24,36,41,1],[16,17,19,21,11,16,3],[15,16,20,22,15,20,6],
[15,16,18,20,15,20,2],[15,16,16,18,15,20,1],[18,19,19,21,38,43,1],[18,19,17,19,39,44,2],[18,19,19,21,8,13,3],[14,15,19,21,19,24,1],
[14,15,16,18,29,34,3],[16,17,17,19,38,43,2],[18,19,21,23,39,44,3],[14,15,17,19,34,39,4],[15,16,17,19,35,40,1],[17,18,16,18,40,45,4],
[17,18,18,20,42,47,5],[17,18,14,16,37,42,2],[16,17,19,21,41,46,4],[16,17,15,17,39,44,6],[15,16,18,20,40,45,3],[16,17,17,19,43,48,1],
[17,18,20,22,42,47,2],[18,19,19,21,43,48,6],[16,17,21,23,39,44,3],[18,19,21,23,9,14,6],[17,18,22,24,11,16,1],[16,17,23,25,35,40,6],
[14,15,18,20,39,44,2],[15,16,20,22,40,45,2],[15,16,21,23,35,40,1],[17,18,18,20,7,12,6],[17,18,20,22,7,12,4],[18,19,15,17,41,46,3],
[18,19,17,19,44,49,1],[16,17,21,23,44,49,6],[15,16,22,24,40,45,5],[14,15,20,22,39,44,3],[16,17,15,17,9,14,3],[17,18,16,18,5,10,5],
[15,16,16,18,40,45,4],[17,18,16,18,45,50,6],[16,17,15,17,44,49,3],[17,18,14,16,42,47,1],[17,18,18,20,47,52,2],[16,17,19,21,46,51,3],
[17,18,20,22,47,52,5],[15,16,15,17,35,40,5],[16,17,17,19,48,53,5],[17,18,14,16,7,12,2],[14,15,16,18,39,44,3],[15,16,18,20,45,50,2],
[14,15,17,19,44,49,5],[15,16,16,18,45,50,6],[14,15,19,21,44,49,1],[15,16,20,22,45,50,5],[14,15,15,17,34,39,1],[15,16,14,16,30,35,4]
,[16,17,13,15,31,36,2],[16,17,13,15,26,31,3],[17,18,12,14,30,35,5],[18,19,13,15,28,33,2],[18,19,13,15,33,38,1],[19,20,14,16,30,35,5],
[19,20,14,16,25,30,1],[18,19,13,15,38,43,5],[18,19,13,15,23,28,5],[18,19,13,15,18,23,4],[16,17,13,15,21,26,1],[18,19,13,15,13,18,1],
[17,18,12,14,35,40,6],[17,18,12,14,20,25,6],[16,17,13,15,16,21,4],[15,16,14,16,18,23,5],[15,16,14,16,13,18,6],[14,15,15,17,17,22,4],
[17,18,12,14,15,20,3],[16,17,13,15,11,16,5],[16,17,13,15,36,41,1],[14,15,15,17,22,27,2],[18,19,13,15,8,13,3],[17,18,12,14,10,15,4],
[18,19,11,13,32,37,3],[18,19,11,13,27,32,1],[19,20,12,14,29,34,6],[14,15,21,23,44,49,4],[16,17,15,17,49,54,2],[17,18,12,14,40,45,3],
[16,17,13,15,41,46,5],[15,16,14,16,40,45,2],[15,16,14,16,45,50,1],[15,16,13,15,35,40,3],[15,16,12,14,30,35,1],[15,16,13,15,25,30,6],
[15,16,12,14,20,25,2],[16,17,11,13,12,17,2],[15,16,12,14,15,20,1],[14,15,13,15,19,24,6],[14,15,13,15,14,19,3],[14,15,14,16,29,34,2]];
\end{verbatim}}

\subsection{\tt431/1/6}\label{431}
{\tiny\begin{verbatim}
[[10,11,12,15,13,17,6],[11,12,13,16,11,15,3],[10,11,12,15,9,13,1],[11,12,10,13,11,15,5],[11,12,11,14,15,19,4],[10,11,9,12,12,16,2],
[12,13,12,15,12,16,1],[11,12,14,17,15,19,2],[10,11,15,18,12,16,4],[12,13,9,12,14,18,2],[11,12,8,11,15,19,3],[12,13,12,15,16,20,3],
[12,13,9,12,18,22,1],[13,14,10,13,15,19,6],[12,13,15,18,14,18,6],[11,12,16,19,11,15,1],[10,11,15,18,8,12,2],[9,10,13,16,10,14,5],
[13,14,13,16,13,17,4],[12,13,15,18,10,14,5],[11,12,14,17,7,11,4],[12,13,12,15,8,12,2],[13,14,10,13,11,15,5],[13,14,13,16,9,13,3],
[12,13,9,12,10,14,4],[11,12,11,14,7,11,6],[11,12,7,10,11,15,1],[10,11,9,12,8,12,3],[9,10,10,13,11,15,4],[10,11,12,15,5,9,5],
[10,11,6,9,10,14,6],[9,10,7,10,13,17,3],[10,11,9,12,16,20,1],[10,11,6,9,14,18,4],[11,12,10,13,19,23,2],[10,11,12,15,17,21,3],
[9,10,10,13,15,19,5],[9,10,13,16,14,18,2],[10,11,15,18,16,20,1],[10,11,6,9,18,22,6],[11,12,7,10,19,23,5],[11,12,5,8,15,19,2],
[12,13,6,9,13,17,5],[12,13,6,9,17,21,4]]
\end{verbatim}}

\subsection{\tt222/1/6}\label{222}
{\tiny\begin{verbatim}
[[14,16,11,13,11,13,1],[11,13,10,12,9,11,2],[12,14,10,12,11,13,6],[13,15,10,12,9,11,3],[12,14,12,14,10,12,4],[14,16,9,11,11,13,4],
[12,14,8,10,10,12,1],[14,16,7,9,11,13,3],[14,16,8,10,9,11,6],[12,14,8,10,8,10,4],[13,15,6,8,9,11,2],[15,17,10,12,9,11,2],
[16,18,8,10,10,12,1],[14,16,12,14,9,11,6],[16,18,10,12,11,13,6],[11,13,6,8,9,11,3],[10,12,8,10,9,11,6],[16,18,12,14,10,12,4],
[14,16,13,15,11,13,5],[12,14,12,14,8,10,1],[13,15,14,16,9,11,3],[14,16,9,11,7,9,1],[14,16,11,13,7,9,5],[12,14,10,12,7,9,6],
[10,12,9,11,7,9,1],[10,12,7,9,7,9,2],[12,14,8,10,6,8,3],[14,16,7,9,7,9,5],[12,14,6,8,7,9,1],[15,17,6,8,9,11,4],[16,18,8,10,8,10,3],
[16,18,10,12,7,9,6],[16,18,12,14,8,10,1],[15,17,14,16,9,11,2],[14,16,13,15,7,9,4],[12,14,12,14,6,8,3],[10,12,11,13,7,9,5],
[11,13,10,12,5,7,4],[13,15,10,12,5,7,2],[10,12,12,14,9,11,3],[9,11,10,12,9,11,4],[14,16,8,10,5,7,6],[16,18,8,10,6,8,2],
[16,18,12,14,6,8,2],[14,16,12,14,5,7,1],[10,12,8,10,5,7,5],[14,16,15,17,7,9,6],[12,14,14,16,7,9,5],[13,15,14,16,5,7,2],
[11,13,14,16,9,11,6],[10,12,13,15,7,9,4],[11,13,14,16,5,7,1],[8,10,8,10,8,10,5],[8,10,8,10,6,8,4],[9,11,10,12,5,7,2],
[8,10,10,12,7,9,3],[8,10,12,14,8,10,2],[12,14,16,18,8,10,1],[8,10,8,10,4,6,3],[12,14,16,18,6,8,3],[10,12,15,17,7,9,2],
[8,10,12,14,6,8,1],[7,9,10,12,5,7,6],[9,11,14,16,9,11,1],[8,10,14,16,7,9,3],[8,10,10,12,3,5,1],[9,11,14,16,5,7,6],
[9,11,12,14,4,6,5],[10,12,16,18,5,7,5],[10,12,16,18,9,11,5],[10,12,17,19,7,9,6],[8,10,16,18,8,10,4],[8,10,16,18,6,8,1],
[7,9,14,16,5,7,5],[7,9,12,14,4,6,2]];
\end{verbatim}}

\subsection{\tt611/2/6}\label{611}
{\tiny\begin{verbatim}
[[11,17,7,8,0,1,3],[12,18,8,9,0,1,5],[11,12,3,9,1,2,5],[12,13,1,7,1,2,3],[12,13,7,13,1,2,1],[13,14,1,7,1,2,5],[13,14,7,13,1,2,4],
[14,15,4,10,1,2,1],[15,16,1,7,1,2,3],[15,16,7,13,1,2,4],[16,17,5,11,1,2,1],[8,14,7,8,2,3,3],[9,15,6,7,2,3,4],[9,15,8,9,2,3,2],
[10,16,10,11,2,3,3],[11,17,4,5,2,3,4],[11,17,5,6,2,3,2],[12,18,9,10,2,3,5],[14,20,7,8,2,3,5],[15,21,6,7,2,3,6],[15,21,8,9,2,3,6],
[12,13,5,11,3,4,1],[13,14,5,11,3,4,6],[15,16,3,9,3,4,3],[16,17,4,10,3,4,1]];
\end{verbatim}}

\subsection{\tt511/2/6}\label{511}
{\tiny\begin{verbatim}
[[8,9,17,22,12,13,3],[11,12,17,22,14,15,4],[10,11,18,23,12,13,2],[7,12,21,22,13,14,6],[7,12,20,21,13,14,5],[9,10,20,25,12,13,1],
[7,12,19,20,13,14,6],[8,13,18,19,13,14,5],[9,10,15,20,12,13,4],[11,12,18,23,12,13,4],[10,11,18,23,14,15,2],[9,14,22,23,13,14,5],
[9,10,18,23,14,15,1],[8,9,18,23,14,15,4],[8,13,17,18,13,14,1],[7,8,18,23,14,15,1],[7,12,21,22,15,16,5],[6,11,22,23,15,16,3],
[7,8,14,19,13,14,4],[7,8,18,23,12,13,1],[4,9,22,23,13,14,2],[7,12,19,20,11,12,6],[7,12,18,19,11,12,5],[7,12,21,22,11,12,6],
[6,11,23,24,14,15,3],[7,12,22,23,11,12,5],[8,9,22,27,12,13,4],[12,13,18,23,11,12,4],[8,13,23,24,13,14,6],[7,12,23,24,10,11,6],
[12,13,15,20,12,13,3],[5,10,23,24,11,12,2],[11,12,22,27,14,15,1],[11,16,22,23,15,16,4],[12,13,20,25,14,15,2],[10,15,23,24,15,16,6],
[8,13,24,25,13,14,3],[12,13,20,25,12,13,1],[8,13,16,17,13,14,2],[10,11,13,18,12,13,3],[11,12,13,18,12,13,6],[10,15,23,24,11,12,3],
[6,11,24,25,11,12,6],[10,11,23,28,12,13,4],[11,12,23,28,12,13,5],[13,14,19,24,14,15,1],[8,13,25,26,13,14,6],[13,14,15,20,13,14,4],
[13,14,23,28,13,14,4],[13,14,19,24,12,13,6],[8,13,26,27,13,14,3],[12,17,21,22,13,14,4],[12,17,20,21,13,14,3],[7,8,23,28,13,14,1],
[7,8,23,28,12,13,3],[10,11,24,29,14,15,2],[8,13,27,28,13,14,6],[9,10,25,30,12,13,2],[7,12,25,26,11,12,1],[9,10,24,29,14,15,1],
[9,10,23,28,15,16,2],[7,12,26,27,11,12,6],[6,11,28,29,13,14,3],[11,16,24,25,11,12,2],[12,13,25,30,14,15,4],[8,13,28,29,15,16,6],
[13,14,24,29,14,15,6],[10,15,27,28,15,16,1],[11,12,27,32,14,15,3],[11,16,28,29,13,14,1],[12,13,25,30,12,13,4],[13,14,24,29,12,13,5],
[9,14,27,28,11,12,1],[14,15,20,25,12,13,1],[14,15,22,27,13,14,2],[10,11,24,29,10,11,5],[11,12,24,29,10,11,4],[12,13,21,26,10,11,1],
[9,10,24,29,10,11,3],[9,14,28,29,11,12,6],[10,11,28,33,12,13,1],[11,12,28,33,12,13,3],[8,13,29,30,11,12,5],[14,19,27,28,13,14,3],
[14,15,25,30,12,13,6],[12,17,25,26,11,12,3],[12,17,26,27,11,12,2],[15,16,24,29,12,13,5],[13,14,23,28,10,11,4],[12,13,26,31,10,11,3],
[15,16,19,24,12,13,6],[14,15,22,27,10,11,1],[13,18,22,23,11,12,2],[13,18,21,22,11,12,5],[16,17,21,26,12,13,1],[10,15,24,25,9,10,3],
[9,14,25,26,9,10,2],[15,16,22,27,10,11,4]]
\end{verbatim}}

\subsection{\tt411/2/5}\label{411-2}
{\tiny\begin{verbatim}
[[4,5,13,17,5,6,5],[1,5,11,12,6,7,1],[4,5,9,13,5,6,3],[3,4,11,15,5,6,2],[2,6,14,15,6,7,4],[2,3,11,15,5,6,3],[2,6,13,14,6,7,1],
[2,6,12,13,6,7,4]]
\end{verbatim}}

\subsection{\tt311/2/5}\label{311-2}
{\tiny\begin{verbatim}
[[9,10,12,15,11,12,1],[13,14,13,16,11,12,4],[10,13,13,14,11,12,2],[10,11,14,17,11,12,3],[9,12,14,15,10,11,2],[11,14,13,14,10,11,3],
[8,11,13,14,10,11,4],[11,14,15,16,10,11,1],[12,15,14,15,10,11,5],[11,12,14,17,11,12,4],[12,13,14,17,11,12,3]];
\end{verbatim}}

\subsection{\tt211/2/5}\label{211}
{\tiny\begin{verbatim}
[[2,3,5,7,2,3,4],[2,3,4,6,1,2,1],[4,6,3,4,1,2,1],[3,5,5,6,1,2,4],[2,4,3,4,3,4,4],[3,5,4,5,1,2,3],[5,6,3,5,2,3,3],[4,6,7,8,2,3,2],
[4,6,3,4,3,4,1],[5,6,4,6,3,4,2],[3,5,7,8,3,4,4],[2,3,6,8,3,4,5],[3,4,2,4,2,3,1],[2,4,6,7,1,2,2],[2,4,7,8,2,3,3],[4,5,4,6,2,3,2],
[2,4,3,4,1,2,4],[5,6,5,7,2,3,4],[4,5,2,4,2,3,4],[3,5,4,5,3,4,3],[3,4,4,6,2,3,5],[2,3,4,6,3,4,1],[3,5,5,6,4,5,4],[2,3,5,7,4,5,3],
[5,6,6,8,3,4,1],[3,5,6,7,4,5,1],[3,4,5,7,3,4,2],[3,5,2,3,3,4,3],[4,6,6,7,1,2,3],[5,6,4,6,1,2,2],[3,5,6,7,2,3,1],[2,3,3,5,2,3,2],
[5,6,5,7,4,5,3],[4,6,4,5,4,5,1],[2,4,4,5,4,5,2],[4,5,5,7,3,4,5]];
\end{verbatim}}

\subsection{\tt421/2/6}\label{421}
{\tiny\begin{verbatim}
[[1,3,3,7,0,1,6],[3,5,3,7,0,1,4],[0,4,1,3,0,1,3],[5,7,3,7,0,1,3],[4,8,1,3,0,1,2],[2,6,2,4,1,2,5],[4,8,6,8,1,2,2],[2,6,4,6,1,2,1],
[6,8,2,6,1,2,4],[0,4,6,8,1,2,3],[0,2,2,6,1,2,2],[5,7,3,7,2,3,3],[1,5,5,7,2,3,5],[1,5,3,5,2,3,4],[1,5,0,2,1,2,6],[4,8,1,3,2,3,2],
[0,4,1,3,2,3,1],[2,6,7,9,2,3,6],[6,8,7,11,2,3,4],[7,9,3,7,2,3,1],[4,8,6,8,3,4,2],[4,6,2,6,3,4,1],[6,8,2,6,3,4,4]]
\end{verbatim}}

\subsection{\tt421/2/6 (alternate)}\label{421alt}
{\tiny\begin{verbatim}
[[0,2,0,4,3,4,1],[2,4,0,4,3,4,3],[4,6,0,4,3,4,1],[1,5,1,3,2,3,4],[1,5,-1,1,2,3,2],[1,5,3,5,2,3,5],[-1,1,0,4,2,3,3],
[5,7,0,4,2,3,3],[2,4,-2,2,1,2,5],[2,4,2,6,1,2,2],[0,2,0,4,1,2,1],[4,6,0,4,1,2,1],[1,5,5,3,-2,-1,6],[1,5,3,1,-2,-1,5],
[1,5,1,-1,-2,-1,6],[2,4,4,0,-1,0,1],[0,2,4,0,-1,0,4],[4,6,4,0,-1,0,3],[1,5,6,4,-1,0,5],[1,5,0,-2,-1,0,5],[-1,3,3,1,0,1,3],
[3,7,3,1,0,1,4],[1,5,5,3,0,1,6],[1,5,1,-1,0,1,6]]
\end{verbatim}}

\subsection{\tt212/2/6}\label{212}
{\tiny\begin{verbatim}
[[14,15,9,11,10,12,5],[11,12,9,11,8,10,5],[13,14,8,10,7,9,4],[12,14,9,10,9,11,1],[12,13,9,11,7,9,6],[13,15,10,11,8,10,2],
[11,13,8,9,8,10,3],[14,16,9,10,8,10,6],[13,15,8,9,9,11,2],[12,13,10,12,9,11,4],[13,14,10,12,10,12,6],[13,14,8,10,11,13,4],
[12,13,9,11,11,13,3],[12,13,7,9,10,12,5],[11,12,8,10,10,12,2],[10,12,10,11,10,12,1],[10,11,8,10,9,11,6],[13,15,10,11,12,14,2],
[10,12,7,8,9,11,4],[12,14,7,8,8,10,6],[14,15,7,9,7,9,3],[13,14,6,8,10,12,3],[12,13,7,9,6,8,5],[11,13,11,12,11,13,5],
[10,12,11,12,9,11,3],[9,11,10,11,8,10,2],[14,16,7,8,9,11,4],[14,15,7,9,11,13,6],[15,16,8,10,10,12,1],[15,17,8,9,8,10,5],
[14,16,9,10,12,14,3],[13,15,6,7,8,10,5],[13,14,6,8,6,8,1],[11,13,6,7,9,11,1],[12,13,5,7,7,9,3],[11,12,6,8,7,9,2],
[11,12,8,10,6,8,1],[10,11,7,9,7,9,5],[9,11,9,10,7,9,3],[11,13,11,12,7,9,2],[13,15,8,9,5,7,2],[14,15,9,11,6,8,5],[13,14,10,12,6,8,3],
[13,14,11,13,8,10,1],[15,16,8,10,6,8,4],[14,15,11,13,7,9,4],[14,16,11,12,9,11,3],[15,16,10,12,7,9,1],[15,17,10,11,9,11,4],
[15,16,10,12,11,13,6],[11,13,12,13,8,10,5],[12,14,12,13,10,12,3],[14,15,11,13,11,13,1],[14,16,12,13,9,11,2],[13,14,11,13,12,14,4],
[16,17,9,11,7,9,2],[16,17,11,13,8,10,6],[16,17,11,13,10,12,1],[16,17,9,11,11,13,2],[15,17,10,11,13,15,4],[14,16,11,12,13,15,3]];
\end{verbatim}}

\subsection{\tt312/2/6}\label{312}
{\tiny\begin{verbatim}
[[4,7,3,4,2,4,6],[4,5,3,6,0,2,2],[3,4,4,7,0,2,5],[1,4,3,4,1,3,1],[2,5,5,6,2,4,1],[3,6,4,5,2,4,3],[5,6,5,8,3,5,4],[0,3,4,5,1,3,6],
[6,7,4,7,3,5,5],[6,9,4,5,5,7,4],[5,6,2,5,4,6,1],[5,8,5,6,5,7,2],[4,5,3,6,4,6,5],[4,7,6,7,5,7,6],[3,4,1,4,3,5,4],[3,4,4,7,4,6,2],
[2,3,2,5,3,5,5],[3,6,4,5,6,8,3],[1,4,3,4,5,7,1],[4,7,3,4,6,8,2],[2,5,2,3,5,7,3],[4,5,0,3,3,5,2],[3,6,1,2,5,7,6],[2,5,5,6,6,8,4],
[3,4,1,4,7,9,4]]
\end{verbatim}}

\subsection{\tt412/2/6}\label{412}
{\tiny\begin{verbatim}
[[3,7,6,7,1,3,1],[2,3,0,4,1,3,5],[2,3,4,8,1,3,6],[3,4,2,6,0,2,4],[4,5,2,6,0,2,6],[5,6,2,6,0,2,4],[6,7,2,6,0,2,5],[3,7,1,2,1,3,1],
[3,7,5,6,2,4,2],[3,7,4,5,2,4,3],[3,7,2,3,2,4,3],[3,7,3,4,2,4,2],[2,3,2,6,3,5,1],[0,4,6,7,3,5,4],[3,4,2,6,4,6,5],[0,4,1,2,3,5,4],
[4,5,2,6,4,6,4],[5,6,2,6,4,6,1],[6,7,2,6,4,6,4],[4,8,1,2,3,5,2],[7,8,2,6,3,5,1],[7,8,1,5,1,3,4],[7,8,5,9,1,3,6],[3,7,4,5,6,8,3],
[3,7,3,4,6,8,2],[7,8,1,5,5,7,6]]
\end{verbatim}}

\subsection{\tt411/3/6}\label{411-3}
{\tiny\begin{verbatim}
[[14,15,28,29,5,9,5],[13,14,25,29,13,14,3],[14,15,28,29,12,16,2],[13,14,28,29,14,18,4],[11,15,27,28,14,15,1],[12,13,28,29,11,15,2],
[12,13,27,28,10,14,6],[13,14,27,31,12,13,4],[11,15,29,30,13,14,1],[11,12,25,29,13,14,3],[10,14,26,27,12,13,5],
[11,12,27,31,12,13,4],[12,13,26,27,13,17,4],[10,11,27,28,11,15,6],[11,12,26,30,11,12,3],[12,13,29,30,9,13,5],[13,14,26,30,11,12,3],
[10,14,30,31,11,12,6],[11,15,28,29,10,11,4],[14,15,27,28,10,14,5],[10,11,28,29,10,14,1],[11,15,29,30,14,15,3],[8,12,28,29,14,15,4],
[10,11,29,30,11,15,5],[14,15,28,32,11,12,1],[8,12,29,30,10,11,6],[8,12,27,28,10,11,2],[11,15,26,27,10,11,4],[13,14,27,28,7,11,2],
[14,15,26,30,9,10,1],[12,13,25,26,10,14,2],[12,13,25,29,9,10,1],[11,12,26,30,9,10,3],[11,12,25,26,9,13,6],[10,11,26,27,8,12,1],
[10,11,27,31,9,10,5],[13,17,29,30,10,11,2],[13,14,28,32,9,10,3],[10,14,28,29,8,9,4],[9,13,27,28,8,9,6],[10,14,30,31,10,11,1],
[11,15,26,27,8,9,5],[11,15,29,30,8,9,2],[14,15,30,31,7,11,4],[12,13,26,30,7,8,3],[13,14,28,32,7,8,1],[14,18,29,30,7,8,3],
[15,16,26,30,11,12,6],[15,16,25,29,12,13,3],[14,18,29,30,12,13,5],[15,16,27,31,13,14,6],[15,16,30,31,9,13,3],[15,16,28,32,8,9,1],
[16,17,28,29,10,14,1],[15,16,26,30,9,10,4],[14,18,27,28,8,9,6]]
\end{verbatim}}
\newcommand{\etalchar}[1]{$^{#1}$}
\providecommand{\bysame}{\leavevmode\hbox to3em{\hrulefill}\thinspace}
\providecommand{\MR}{\relax\ifhmode\unskip\space\fi MR }
\providecommand{\MRhref}[2]{%
  \href{http://www.ams.org/mathscinet-getitem?mr=#1}{#2}
}
\providecommand{\href}[2]{#2}

\noindent
(S. Eilers): {\sc Department of Mathematical Sciences, University of Copenhagen, Universitetsparken 5, DK-2100 Copenhagen \O, Denmark}\\
\emph{Email address :} \verb!eilers@math.ku.dk!\\[4mm]

\noindent
(R. Johansen): {\sc V\ae{}rl\o{}se, Denmark}\\
\emph{Email address :} \verb!rune@cyx.dk!\\[4mm]

\noindent
(R.V. Rasmussen):  {\sc Department of Mathematical Sciences, University of Copenhagen, Universitetsparken 5, DK-2100 Copenhagen \O, Denmark}\\
\emph{Current address :} {\sc Copenhagen, Denmark}\\
\emph{Email address :} \verb!rasmusrasmussen1998@gmail.com!\\[4mm]

\noindent
(C. Thomassen):  {\sc 
Department of Applied Mathematics and Computer Science, Richard Petersens Plads, 322, DK-2800 Kgs. Lyngby,
Denmark\\
\emph{Email address :} \verb!ctho@dtu.dk!
\end{document}

\newcommand{\MR}[1]{}
\newcommand{\etalchar}[1]{$^{#1}$}
\providecommand{\bysame}{\leavevmode\hboxto3em{\hrulefill}\thinspace}
\providecommand{\MR}{\relax\ifhmode\unskip\space\fi MR }
\providecommand{\MRhref}[2]{%
  \href{http://www.ams.org/mathscinet-getitem?mr=#1}{#2}
}
\providecommand{\href}[2]{#2}

\end{document}